\documentclass[11pt]{amsart}
\usepackage{amsthm,amsmath,amssymb}

\usepackage{thmtools}
\usepackage{thm-restate}
\usepackage{longtable}

\usepackage{graphicx}
\usepackage{color}
\usepackage{dsfont}
\usepackage{float}
\usepackage[colorlinks=true,
linkcolor=blue,
urlcolor=blue,
citecolor=blue]{hyperref}
\usepackage{mathtools}
\usepackage{relsize}
\usepackage{tikz}
\usepackage{enumerate}
\usepackage[toc,page]{appendix}
\usepackage{pdflscape}
\usepackage{multirow}
\usepackage{adjustbox,expl3,etoolbox} 
\usepackage[margin = 2.54cm]{geometry}
\usepackage{mathrsfs}
\usepackage{rotating}
\usepackage{times}

\newtheorem{thm}{Theorem}[section]
\newtheorem{lem}[thm]{Lemma}
\newtheorem{defn}[thm]{Definition}
\newtheorem{prop}[thm]{Proposition}
\newtheorem{conj}[thm]{Conjecture}
\newtheorem{cor}[thm]{Corollary}
\theoremstyle{definition}

\newtheorem{quest}[thm]{Question}

\usepackage[foot]{amsaddr}
\newcommand{\pgl}{\mathop{\mathrm{PGL}}}
\newcommand{\psl}{\mathop{\mathrm{PSL}}}
\newcommand{\pg}{\mathop{\mathrm{PG}}}
\newcommand{\cyc}{\mathop{\mathrm{Cyc}}}
\newcommand{\aut}{\mathop{\mathrm{Aut}}}
\newcommand{\sym}{\mathop{\mathrm{Sym}}}

\newcommand{\qbinom}[2]{\genfrac{[}{]}{0pt}{}{#1}{#2}_q}

\newcommand{\Mod}[1]{ (\mathrm{mod} \ #1)}

\usepackage[hang,flushmargin]{footmisc}

\title{Near-factorizations in association schemes}

\author[A. Herman]{Allen Herman$^1$}
\thanks{$^1$Department of Mathematics and Statistics, University of Regina, Regina, SK, Canada S4S 0A2}

\author[A. Lacaze-Masmonteil]{Alice Lacaze-Masmonteil$^{2,*}$}
\thanks{$^2$Department of Mathematics, University of Auckland, Auckland, New Zealand 1010}
\thanks{$^*$ Corresponding author.}

\author[K. Meagher]{Karen Meagher$^1$}

\author[H. Monterde]{Hermie Monterde$^1$}

\begin{document}

\maketitle

\begin{abstract} 
We formally initiate the study of $\lambda$-fold $(s,t)$ near-factorizations in association schemes. Namely, given an association $(X, \mathcal{A})$, we consider the existence of a factorization of $\lambda(J-I)$ into 01-matrices $S$ and $T$ with the constraint that $S$ and $T$ must belong to the adjacency algebra of $(X, \mathcal{A})$. We establish basic properties of $\lambda$-fold $(s,t)$ near-factorizations in association schemes and calculate bounds for $\lambda$, $s$, and $t$ relative to the order of the association scheme. We completely determine all near-factorizations in symmetric and asymmetric 2-class association schemes. Furthermore, we construct near-factorizations in certain Hamming schemes, cyclotomic association schemes, Schurian schemes on small primitive groups, and the folded cube and halved cube association schemes. Finally, we establish that certain Hamming schemes, Johnson schemes, and Grassmann schemes do not admit a $\lambda$-fold near-factorization.
\end{abstract}

	


\section{Introduction}
\label{S:intro}
The question of existence of a factorization of a group $G$ into two subgroups has a rich history. These factorizations have applications in coding theory, design theory, and cryptography.  A variation of this problem is the existence of a \emph{near-factorization of a group $G$}, which is a pair of subsets $S, T \subset G$ such that, for each $h \in G-e$, there exists a unique pair $(g_s, g_t) \in S \times T$ such that $h=g_s g_t$. It is straightforward to see that the pair $S=\{g\}$ and $T=G-\{g^{-1}\}$ is a near-factorization of $G$ for all $g \in G$. Thus, if $|S|=1$ or $|T|=1$, then $(S,T)$ is said to be a \emph{trivial near-factorization}. Consequently, the literature focuses on the question of existence of nontrivial near-factorizations. This question has been considered for various groups including cyclic groups, dihedral groups, and a few sporadic non-abelian and non-dihedral groups. See the introduction of \cite{kreher2025near} for a brief survey of known near-factorizations. There are also no known near-factorizations of non-cyclic abelian groups. In \cite{unique}, Kreher et al.\ give strong evidence that non-cyclic abelian groups do not admit nontrivial near-factorizations by verifying all non-cyclic abelian groups of order up to 200. 

Recently, Kreher et al.~\cite{kreher2025lambda} extended the concept of a near-factorization in a group to a \emph{$\lambda$-fold near-factorization}, which in this case, is a pair of subsets $S, T \subset G$ such that, for each $h \in G-e$, there exists exactly $\lambda$ pairs $(g_s, g_t) \in S \times T$ for which $h=g_s g_t$. Observe that a $1$-fold near-factorization is a near-factorization of $G$. If $(S,T)$ is a $\lambda$-fold near-factorization of $G$, and $M(S)$ and $M(T)$ are the adjacency matrices of the Cayley graphs of $G$ with connecting sets $S$ and $T$, respectively, then \cite[Theorem 1.3]{kreher2025lambda} implies that $M(S)M(T)=\lambda(J-I)$, where $I$ and $J$ are the identity and all-ones matrix of dimension $|G| \times |G|$, respectively. Thus, a $\lambda$-fold near-factorization of a group can also be viewed as a factorization of the matrix $\lambda(J-I)$. In~\cite{kreher2025lambda}, Kreher et al.\ found all realizations of $\lambda$-fold near-factorizations in abelian groups of order at most 35. In most cases, they were able to generalize these examples to obtain infinite families of $\lambda$-fold near-factorizations in abelian groups. Interestingly, if $\lambda> 1$, results of \cite{kreher2025lambda} imply that there exists examples of $\lambda$-fold near-factorizations in non-cyclic abelian groups. 

Motivated by recent results in the group case, we formally initiate the study of $\lambda$-fold near-factorizations in association schemes. 

\begin{defn}
\label{defn:AS}
Let $X$ be a non-empty set with $|X|=n$. An \emph{association scheme} $(X,\mathcal{A})$ is a finite set of $n \times n$ binary matrices $\mathcal{A}=\{A_0, A_1, \ldots, A_{d}\}$ with rows and columns indexed by $X$ such that: 
\begin{enumerate}
\item $A_0=I_{n}$, the $n\times n$ identity matrix; 
\item $\sum_{i=0}^dA_i=J_n$, the $n\times n$ all-ones matrix;
\item $A_i^T \in \mathcal{A}$ for each $i \in \{0,1,\ldots, d\}$, where $A_i^T$ is the transpose of $A_i$;
\item $A_iA_j=A_jA_i$ for all $0\leq i,j\leq d$; and
\item for all $0\leq i,j\leq d$, the product $A_iA_j$ lies in the linear span of $\{A_0, A_1, \ldots, A_{d}\}$ over $\mathbb{C}$. 
\end{enumerate}
Two points $x_1, x_2\in U$ are said to be \emph{$i$-related} if and only if $A_i(x_1, x_2)=1$. The indices in $\{0,1,\ldots,d\}$ are referred to as the \emph{relations} of the scheme. Each matrix $A_i$ is called a \emph{class} or an \emph{associate}. 
\end{defn}

Observe that, for any two elements $x, y \in X$ there is a unique $A_i \in \mathcal{A}$ with $(x,y)$-entry equal to 1. Moreover, by Axiom (2) of Definition \ref{defn:AS}, the classes of the scheme must form a linearly independent set. By Axiom (5), the product of any two elements in $\mathcal{A}$ belongs to $\operatorname{span}(\mathcal{A})$ and thus, in conjunction with the second axiom, we get that $\mathcal{A}$ is a basis for a commutative $d+1$-dimensional sub-algebra of $M_n(\mathbb{C})$ called the \emph{adjacency algebra} or the \emph{Bose-Mesner algebra} of the association scheme $(X,\mathcal{A})$. The row (or column) sum of $A_i$ is the \emph{valency} of the class; this is well-defined since $J$, being the sum of all classes, commutes with each $A_i$.  An association scheme $(X,\mathcal{A})$ is \emph{symmetric} if each class is symmetric, i.e. $A_i^\top=A_i$ for all $i\in\{0,1,\ldots, d\}$. Since each $A_i$ is a 01-matrix, we may view each $A_i$ as the adjacency matrix of a directed graph $\Gamma_i$ with vertex set $X$. This directed graph is a graph if $A_i$ is symmetric. An \emph{automorphism} of an association scheme $(X,\mathcal{A})$ is a permutation matrix that commutes with each $A_i$. In particular, if a class $A_\ell$ has valency one, then $A_\ell$ is an automorphism of $(X,\mathcal{A})$.

The focus of this paper is the question of existence of a factorization of the matrix $\lambda(J-I)$ into the product of 01-matrices $S$ and $T$ that belong to the adjacency algebra of the scheme. In Section \ref{sec:NF}, we will show that an affirmative answer to this question is equivalent to the existence of a pair of subsets of relations satisfying certain properties made precise in Theorem \ref{thm:char1}. Consequently, this factorization can be viewed as the association scheme analogue of a $\lambda$-fold near-factorization in groups.

Our main motivation for studying near-factorizations in association schemes stem from its connections to combinatorial designs with nice automorphisms. Combinatorial designs such as BIBD can be constructed using difference sets and partial difference sets of a finite group $G$. In \cite{kreher2025lambda}, Kreher et al.\ showed that certain nontrivial $\lambda$-fold  near-factorizations in a group $G$ give rise to difference sets or partial difference sets of $G$. Additionally, Li and Momihara \cite{li-momihara2025} observed that specific $\lambda$-fold near-factorizations in additive groups of finite fields give rise to strong external difference family. We provide several examples of near-factorizations in fusions of thin association schemes that correspond to difference sets and partial difference sets in the corresponding group. 

Although we are the first to formally study $\lambda$-fold near-factorizations in association scheme, we point out that these factorizations have previously been considered in specific schemes under the name of \emph{central difference sets}~\cite{MR2216464} and \emph{central partial difference sets}~\cite{MR4903728}, respectively.  These near-factorizations arise in certain conjugacy class schemes of non-abelian groups and are only known to exist in certain $p$-groups of Suzuki type. We have been able to verify that the smallest ones are the $16$-fold $(28,36)$ difference sets found previously in Suzuki $2$-groups of order $64$ by Gow and Quinlan~\cite{MR2216464}. In searching for near-factorizations in conjugacy class association schemes, the second smallest order we found is $16$-fold $(28,68)$ near-factorization in the conjugacy class scheme of the group $C_2 \times A_5$ of order $120$ (see Subsection~\ref{subsec:CCscheme}).

\subsection{Structure of the paper}

We organize this paper as follows. First, we develop the theory of $\lambda$-fold $(s, t)$ near-factorizations in association schemes in Section~\ref{sec:NF}. In particular, we establish basic properties, provide characterizations in terms of the intersection numbers and eigenvalues of the scheme, and calculate bounds for the parameters $s$, $t$ and $\lambda$. 
In Section~\ref{sec:2class}, we completely characterize near-factorizations in symmetric and asymmetric 2-class schemes. We devote Section~\ref{sec:genhex}, to showing that the four class $P$-polynomial association scheme derived from a generalized hexagon admits a near-factorization. We proceed to Hamming schemes $H(n, q)$ in Section~\ref{sec:Hamming}, where we establish the existence of near-factorizations in $H(n,2)$ for even $n$, and in $H(n, q)$ for $q \in \{3,4\}$ and all $n\geqslant 2$. Additionally, we show that $H(n, q)$ has no nontrivial near-factorizations when $q >n$. While it is still open whether $H(n,2)$ admits a near-factorization when $n$ is odd, we show in Section~\ref{sec:FFHamming} that the folded cube and halved cube association schemes, which are schemes derived from $H(n,2)$, admit near-factorizations for all odd $n\geq 3$. 
In Sections~\ref{sec:Johnson} and \ref{sec:grassman}, we establish that the Johnson scheme $J(n, k)$ and the Grassmann scheme $J_q(n, d)$ do not admit nontrivial near-factorizations when $n$ is sufficiently large relative to $k$ and $d$, respectively. We also investigate near-factorizations in cyclotomic schemes and orbit fusions of thin schemes in Section~\ref{sec:cyclo}. In particular, we establish a connection between the existence of near-factorizations in cyclotomic schemes and certain identities satisfied by Gaussian periods. We also provide $1$-fold near-factorizations in symmetrization schemes of cyclic groups of odd order. In Section~\ref{sec:schurian}, we discuss near-factorizations occurring in small (non-thin) Schurian schemes. These near-factorizations are summarized in Table~\ref{Table:Schurian}. Finally, we conclude with some open questions in Section~\ref{sec:further}.


\section{The near-factorization condition for association schemes}\label{sec:NF}

Let $(X,\mathcal{A})$ be an association scheme. Since the $d$ classes of $(X,\mathcal{A})$ pairwise commute and the product of any two elements in $\mathcal{A}$ belongs to the adjacency algebra of $(X,\mathcal{A})$, there exists a set of nonnegative integers $\{p_{i,j}(k):0\leqslant i,j,k\leqslant d\}$ called \emph{intersection numbers} satisfying 

\begin{equation}
\label{intnumbers}
A_iA_j=A_jA_i=\sum_{k=0}^dp_{i,j}(k)A_k.
\end{equation}
The \emph{intersection matrices} $B_0,B_1,\ldots,B_d$ of $(X,\mathcal{A})$ satisfy
\begin{center}
$(B_i)_{k,j}:=p_{i,j}(k).$
\end{center}

The set of intersection matrices $\{B_0,B_1,\ldots,B_d\}$ generates a commutative algebra of $(d+1)\times (d+1)$ matrices called the \emph{intersection algebra}.

Let $(X,\mathcal{A})$ be an association scheme and $J$ be the all-ones matrix of the appropriate size. Since the elements of $\mathcal{A}$ are
all symmetric and pairwise commuting matrices, it follows that the classes are simultaneously diagonalizable. That is, there exist orthogonal projection matrices $E_0:=\frac{1}{n}J,\ldots,E_d$ and scalars $\{p_i(j):0\leqslant i,j\leqslant d\}$ such that for all $i\in \{0,1,\ldots,d\}$, we have
\begin{equation}
\label{evals}
A_i=\displaystyle\sum_{r=0}^dp_i(r)E_r.
\end{equation}
The scalars $p_i(j)$ are called the \emph{eigenvalues} of the association scheme. In particular, we say that $p_i(r)$ is an eigenvalue of $A_i$ with associated orthogonal projection matrix $E_r$. Note that$\{E_0,\ldots,E_d\}$ also form a basis for the adjacency algebra. By Axiom (2) of Definition \ref{defn:AS}, the eigenvalue of $A_i$ realized by the trivial eigenspace is the valency of $A_i$. We generally denote this eigenvalue as $p_i(0)$. Otherwise, also by Axiom (2), the eigenvalues realized by a non-trivial eigenspace must sum to zero. 

Let $(X_1,\mathcal{A}_1)$ and $(X_2,\mathcal{A}_2)$ be association schemes. We say that $(X_2,\mathcal{A}_2)$ is a \emph{fusion} of $(X_1,\mathcal{A}_1)$ if the adjacency algebra of $(X_2,\mathcal{A}_2)$ is a subspace of that of $(X_1,\mathcal{A}_1)$. If $(X_2,\mathcal{A}_2)$ is a fusion of $(X_1,\mathcal{A}_1)$, then $(X_1,\mathcal{A}_1)$ is a \emph{fission} of $(X_2,\mathcal{A}_2)$. A fusion $(X_2,\mathcal{A}_2)$ of an association scheme $(X_1,\mathcal{A}_1)$ arises from $A_0=I$ and a partition of $\mathcal{A}_1\backslash\{A_0\}=\{A_1,\dots,A_d\}$, where the sums of the matrices in each part of the partition satisfy the definition of an association scheme. That is, a fusion scheme is obtained by merging certain relations of a base scheme, whereas a fission scheme is a refinement of a coarser (fusion) scheme. For more about association schemes, we refer the reader to \cite{bannai2021algebraic}. 

We now define the $\lambda$-fold near-factorization analogue for association scheme. 

\begin{defn}
Let $\lambda, s, t \in \mathbb{Z}^+$. A \emph{$\lambda$-fold $(s,t)$ near-factorization} in an association scheme $(X,\mathcal{A})$ is pair of $01$-matrices $S$ and $T$ in the adjacency algebra with valencies $s$ and $t$, respectively, such that
\begin{equation}\label{eq:NF}
ST = \lambda(J-I) = \lambda(A_1 + \dots + A_d).
\end{equation}  
\end{defn}

For brevity, if $\lambda$, $s$ and $t$ are not needed, then we simply say \emph{near-factorization}. 

Any association scheme with a class $A_i$ of valency one admits a $1$-fold $(1,n-1)$ near-factorization realized by $S=A_i$ and $T=J-A_i$; a near-factorization with either $s=1$ or $t=1$ will be called \emph{trivial}. In particular, since $A_0=I$, every association scheme has a trivial near-factorization with $S=I$ and $T=J-I$.

\subsection{Basic properties of near-factorizations}

We now discuss some basic properties of $\lambda$-fold $(s,t)$ near-factorizations in association schemes.

\begin{prop}
\label{prop:aut}
Suppose $(X,\mathcal{A})$ has a $\lambda$-fold $(s,t)$ near-factorization as in (\ref{eq:NF}). If $P$ is an automorphism of $(X,\mathcal{A})$, then $ST=(SP)(P^\top T)=P(ST)P^\top=\lambda(J-I)$.
\end{prop}

\begin{proof}
This is immediate from the fact that $P$ commutes with $S$ and $T$ and $P^\top P=PP^\top=I$.
\end{proof}

By Proposition~\ref{prop:aut}, near-factorizations are equivalent under the action of automorphisms.  Next we show 
a type of uniqueness for $\lambda$-fold $(s,t)$ near-factorizations in association schemes. The following result is analogous to \cite[Theorem 1.6]{kreher2025lambda}, and uses a result analogous to~\cite[Theorem 1.3]{kreher2025lambda}. 

\begin{thm}
\label{thm:uniqueness}
If $(X,\mathcal{A})$ admits a $\lambda$-fold $(s,t)$ near-factorization realized by the pair $(S,T)$ and a $\lambda'$-fold $(s,t')$ near-factorization realized by the pair $(S,T')$, then $\lambda=\lambda'$, $T=T'$ and $t=t'$.
\end{thm}

\begin{proof}
Since $ST=\lambda(J-I)$, and $J-I$ is nonsingular, it must be that $S$ and $T$ are both nonsingular. Since $SJ=sJ$, we get that $S^{-1}J=\frac{1}{s}J$. Hence, $T=\lambda S^{-1}(J-I)=\frac{\lambda}{s}J-S^{-1}$ thereby implying that $T=\frac{\lambda}{s}J-\lambda S^{-1}$ and $T'=\frac{\lambda'}{s}J-\lambda' S^{-1}$. Consequently, $\frac{1}{\lambda}T=\frac{1}{\lambda'}T'$. Since $T$ and $T'$ are 01-matrices, it must be that $T=T'$ and $\lambda=\lambda'$, and so $t=t'$.
\end{proof}

The above proof also yields a formula for $T$ in terms of $S$, which also provides further constraints on the 01-matrices forming a near-factorization.

\begin{cor}
\label{cor:prop}
Suppose $(X,\mathcal{A})$ admits a $\lambda$-fold $(s,t)$ near-factorization realized by the pair $(S,T)$. Then 
\begin{enumerate}[(i)]
\item $T = \frac{\lambda}{s} J - \lambda S^{-1}$; and
\item if $(X,\mathcal{A})$ also admits a $\lambda$-fold $(s,t)$ near-factorization realized by $(S,T')$, then $T=T'$;
\item $S$ is symmetric if and only if $T$ is symmetric. \label{symmetry}
\end{enumerate}
\end{cor} 

If $(X,\mathcal{A})$ is symmetric, then $S$ and $T$ are also symmetric. For an infinite family of schemes admitting nontrivial near-factorizations, where $S$ and $T$ are not symmetric, see Theorem~\ref{thm:Asym2class}. 

\subsection{Characterizations of near-factorizations}

We now provide two characterizations for the existence of $\lambda$-fold $(s,t)$ near-factorization in association schemes. The first one involves the intersection numbers $p_{i,j}(k)$ of the association scheme.

\begin{thm}
\label{thm:char1}
An association scheme with $d\geqslant 2$ classes admits a $\lambda$-fold $(s,t)$ near-factorization $(S,T)$
if and only if there exist nonempty sets of indices $U,V \subseteq \{0,1,\dots,d\}$ such that

\begin{enumerate}
\item
\begin{equation*}
S=S_U:=\sum_{i \in U} A_i   \quad \text{and} \quad T=T_V:=\sum_{j \in V} A_j.
\end{equation*}

\item For all $i\in U$ and $j\in V$, we have $p_{i,j}(0)=0$.

\item For all $k \in \{1,\ldots, d\}$,
$$
\displaystyle\sum_{j\in V}\sum_{i\in U}\sum_{k\geqslant 1}^dp_{i,j}(k) =\lambda.
$$
\end{enumerate}
\end{thm}

\begin{proof}
Suppose $(X,\mathcal{A})$ admits a $\lambda$-fold $(s, t)$ near-factorization $ST=\lambda(J-I)$. Since $S$ and $T$ are $01$-matrices in the adjacency algebra of the association scheme and the classes form a linearly independent set, it must be that $S=\sum_{i \in U} A_i$ and $T=\sum_{j \in V} A_j$ for nonempty subsets $U,V \subseteq \{0,1,\dots,d\}$. 

Using Statement (1), we obtain 
\begin{equation}
\label{22}
ST=\bigg(\sum_{i \in U} A_i\bigg)\bigg(\sum_{j \in V} A_j\bigg)=\sum_{i \in U} \sum_{j \in V} A_iA_j=\sum_{i \in U} \sum_{j \in V}\bigg(p_{i,j}(0)I+\sum_{k\geqslant 1}^dp_{i,j}(k)A_k\bigg).
\end{equation}
Since $ST=\lambda(A_1+\cdots+A_d)$, the classes form a linearly independent set, and the parameters $p_{i,j}(k)$ are all nonnegative, comparing corresponding classes in (\ref{22}) implies statements (2) and (3). For the converse, it is easy to see that if conditions (1), (2) and (3) above are satisfied, then $ST=\lambda(J-I)$.
\end{proof}

Since the classes in an association scheme commute, the two matrices in a near-factorization must also commute by Theorem~\ref{thm:char1}(1). Thus, $(X,\mathcal{A})$ admits a nontrivial $\lambda$-fold $(s,t)$ near-factorization if and only if it admits a nontrivial $\lambda$-fold $(t,s)$ near-factorization, in which case the near-factorizations are realized by the same pair of matrices $S$ and $T$.

\begin{cor}
\label{symmetry2}
Suppose $(X,\mathcal{A})$ admits a $\lambda$-fold $(s,t)$ near-factorization realized by the pair $(S,T)$. If $S$ is not symmetric, then $(S^\top,T^\top)$ is another $\lambda$-fold $(s,t)$ near-factorization in $(X,\mathcal{A})$.
\end{cor} 

\begin{thm}
\label{thm:UV}
Suppose $(X,\mathcal{A})$ admits a $\lambda$-fold $(s,t)$ near-factorization and consider the subsets $U$ and $V$ in Theorem~\ref{thm:char1}(1). If $U'=\{i' : i \in U \ \textrm{and} \ A_{i'} := A_i^{\top}\}$, then $U' \cap V=\varnothing.$ In particular, if $S$ is symmetric, then $U\cap V=\varnothing$.
\end{thm}

\begin{proof}
By way of contradiction, suppose $\ell' \in U'\cap V \neq \varnothing$. Then $A_\ell$ is a class in $S$ while $A_\ell^\top$ in $T$, and therefore $A_\ell A_\ell^\top$ appears as a summand in the product $ST=S_U T_V$. However, $A_\ell A_\ell^\top$ has positive diagonal, and thus, $A_0=I$ appears as a summand in the product $ST$, a contradiction. This shows that $U' \cap V=\varnothing.$ In particular, if $S$ is symmetric, then $U'=U$, and so $U\cap V=\varnothing$.
\end{proof}

The following is immediate from Theorem \ref{thm:char1}(1).

\begin{prop}
\label{prop:fuse}
If $(X,\mathcal{A})$ admits a nontrivial $\lambda$-fold $(s, t)$ near-factorization, then any fission scheme of $(X,\mathcal{A})$ also admits  a nontrivial $\lambda$-fold $(s, t)$ near-factorization realized by the same matrices.
\end{prop}

By virtue of Proposition~\ref{prop:fuse}, we will present our near-factorization results, where it is reasonable to do so, in terms of the minimal fusion where the near-factorization is realized. 

Next, we give our second characterization, this time using the eigenvalues of the scheme. 

\begin{thm}
\label{thm:char2}
An association scheme with $d\geqslant 2$ classes admits a $\lambda$-fold $(s,t)$ near-factorization if and only if there exists a pair of non-empty sets $U, V \subseteq \{1, \ldots,d\}$ such that
\begin{enumerate}
\item 
\begin{equation*}
\sum_{i\in U}\sum_{j\in V}p_i(0)p_j(0) = \lambda(n-1), \ \textrm{ and }
\end{equation*}

\item for all $r\geq 1$, \label{char2:cond3}
$$\displaystyle\sum_{i\in U}\displaystyle\sum_{j\in V}p_i(r)p_j(r) = -\lambda. $$
\end{enumerate}
\end{thm}

\begin{proof}
Suppose $(X,\mathcal{A})$ admits a $\lambda$-fold $(s,t)$ near-factorization $ST=\lambda(J-I)$. From Theorem~\ref{thm:char1}~(1), (1) is a necessary condition. From the form of $S$ and $T$, we deduce from (\ref{evals}) that 
\begin{equation}
\label{STspec}
S=\displaystyle\sum_{i\in U}A_i=\displaystyle\sum_{r=0}^d\bigg(\displaystyle\sum_{i\in U}p_i(r)\bigg)E_r\quad \text{and}\quad T=\displaystyle\sum_{j\in V}A_i=\displaystyle\sum_{r=0}^d\bigg(\displaystyle\sum_{j\in V}p_j(r)\bigg)E_r.
\end{equation}

Thus, for each $r\geq 0$, $\sum_{i\in U}p_i(r)$ and $\sum_{j\in V}p_i(r)$ are eigenvalues of $S$ and $T$, respectively, with associated orthogonal projection matrix $E_r$. In particular, $\sum_{i\in U}p_i(0)$ and $\sum_{j\in V}p_j(0)$ are the valencies of $S$ and $T$, respectively. Making use of Equation (\ref{STspec}), we obtain

\begin{equation}
\label{STspec1}
ST=\displaystyle\sum_{r=0}^d\bigg(\displaystyle\sum_{i\in U}\displaystyle\sum_{j\in V}p_i(r)p_j(r)\bigg)E_r=\lambda(J-I).
\end{equation}
Since $\frac{1}{n}J$ is an orthogonal projection matrix for $S$, $T$ and $J-I$ associated with the eigenvalues $s=\sum_{i\in U}p_i(0)$, $t=\sum_{j\in V}p_j(0)$ and $n-1$ respectively, we get 

\begin{center}
$st(\frac{1}{n}J)=ST(\frac{1}{n}J)=\lambda(J-I)\frac{1}{n}J=\lambda(n-1)\frac{1}{n}J$.
\end{center}
This proves Statement (2). Now, if $r\geq 1$, then multiplying (\ref{STspec1}) by $E_r$ yields 3. The converse is immediate.
\end{proof}

The following is a straightforward consequence of Theorem~\ref{thm:char2} Statements (2) and (3).

\begin{cor}
\label{cor:2ndeval}
If $ST=\lambda(J-I)$, then the largest eigenvalue of $ST$ is $\lambda(n-1)$ and all other eigenvalues of $ST$ are equal to $-\lambda$.
\end{cor}

\subsection{Bounds for $\lambda$, $s$ and $t$}

If $(S,T)$ is a trivial $\lambda$-fold $(s,t)$ near-factorization, then $S = A_i$ for some $A_i$ with valency 1. Therefore $A_i$ is an automorphism and using Proposition~\ref{prop:aut}, we have

$$
T = A_i( A_i T) = A_i \lambda (J-I) = J - A_i.
$$
Thus, in the trivial case, we have $s = \lambda =1$ and $s+t = n$. Next, we establish certain constraints on $s$, $t$ and $s+t$ in terms of $n$ in a nontrivial $(s,t)$ near-factorization.

\begin{thm}
\label{thm:stbound}
Suppose $(X,\mathcal{A})$ admits a nontrivial $\lambda$-fold $(s,t)$ near-factorization. Then
\begin{equation}
\label{st}
\lambda< s, t \leqslant n-2\quad \text{and}\quad 4 \leqslant s+t \leqslant n.
\end{equation}
\end{thm}

\begin{proof}
Suppose $(X,\mathcal{A})$ admits a nontrivial $\lambda$-fold $(s,t)$ near-factorization. Thus, $s,t\geq 2$. Since the sum of all classes is $J$, Theorem~\ref{thm:UV} yields $ 4\leqslant s+t \leqslant n$. If $t=n$, then $T = A_0+\dots+A_d = J$, which gives us $ST = sJ$, a contradiction. Thus, $t\leqslant n-1$, and similarly, $s\leqslant n-1$. However, if $t=n-1$, then the fact that $s+t\leqslant n$ implies that $s=1$, a contradiction. Thus, $t\leqslant n-2$. Similarly, $s\leqslant n-2$. Finally, since $s,t<n-1$, it must be that $\lambda<s, t $. Otherwise, we get $s t<\lambda(n-1)$, a contradiction to Statement (1) of Theorem \ref{thm:char1}.
\end{proof}

Note that the upper bound in Theorem \ref{thm:stbound} for $s+t$ is tight. See Theorems~\ref{thm:SRGs}(1,2) and~\ref{thm:Asym2class} for infinite families of symmetric and asymmetric 2-class schemes that meet this upper bound.

We now rule out certain values of $\lambda$, $s$ and $t$ in a near-factorization.  Proposition \ref{prop:gcd} below is immediate from Statement (1) of Theorem \ref{thm:char2}. Note that the roles of $s$ and $t$ can be switched.

\begin{prop}
\label{prop:gcd}
Suppose $(X,\mathcal{A})$ admits a $\lambda$-fold $(s, t)$ near-factorization. 
If $\operatorname{gcd}(s,\lambda)=1$, then $\lambda$ divides $t$ and $s$ divides $n-1$. 
If $\lambda=1$, then $s$ and $t$ both divide $n-1$ and if $\lambda\geq 2$, then $s$ and $t$ cannot be both prime.
Whenever $s$ is prime, we have $\operatorname{gcd}(t,\lambda)\geq 2$.
\end{prop}

By Proposition~\ref{prop:gcd}, if $(X,\mathcal{A})$ admits a nontrivial $\lambda$-fold near-factorization (so that $s, t>\lambda$ by Theorem~\ref{thm:stbound}) and $s, t$ are both prime, then it must be that $\lambda=1$.

The fact that $s t=\lambda(n-1)$, with $s, t \leqslant n-2$, rules out nontrivial near-factorizations when $n-1$ is prime. 

\begin{cor}
\label{cor:n-1}
If $\operatorname{gcd}(s, n-1)=1$ or $\operatorname{gcd}(t, n-1)=1$, then $(X,\mathcal{A})$ has no nontrivial near-factorization. In particular, if $n-1$ is prime, then $(X,\mathcal{A})$ has no nontrivial near-factorization. 
\end{cor}

By definition, $\lambda \geqslant 1$. In fact, it is possible to obtain nontrivial $1$-fold $(s, t)$ near-factorizations in association schemes (see Subsection~\ref{subsec:SymofThin} for infinite families of examples). Additionally, since $\lambda \geqslant 1$, Theorem~\ref{thm:stbound} implies that $s, t \geqslant 2$. The infinite families in Subsection~\ref{subsec:SymofThin} satisfy $s=2$. Thus, the lower bound 2 for both $s$ and $t$ is tight.

For upper bounds on $\lambda$ in terms of $s$ and $t$, we have the following result from~\cite[Theorem 1.15]{kreher2025lambda} which applies to $\lambda$-fold $(s, t)$ near-factorizations in association schemes.

\begin{thm}
\label{thm:lambdast}
If $(X,\mathcal{A})$ is an association scheme on $n$ points that admits a nontrivial $\lambda$-fold $(s, t)$ near-factorization, then
\begin{center}
$\lambda \leqslant \dfrac{st}{s+t-1}$
\end{center}
with equality if and only if $s+t=n$.
\end{thm}

In Theorem~\ref{thm:lambdast}, the upper bound $\frac{s t}{s+t-1}$ for $\lambda$ is attained by infinite families of symmetric 2-class schemes in Theorem~\ref{thm:SRGs}(1,2). For upper bounds on $\lambda$ in terms of $n$, we have the next result.

\begin{thm}
\label{thm:maxlambda}
If $(X,\mathcal{A})$ is an association scheme on $n$ points that admits a nontrivial $\lambda$-fold $(s,t)$ near-factorization, then
\[
\lambda\leqslant 
  \begin{cases}
    \frac{n}{4}, & \text{if $n\equiv 0 \pmod{4}$},\\
    \frac{n-1}{4}, & \text{if $n\equiv 1 \pmod{4}$},\\
    \frac{n-2}{4}, & \text{if $n\equiv 2 \pmod{4}$},\\
    \frac{n+1}{4}, & \text{if $n\equiv 3 \pmod{4}$}.
  \end{cases}
\]
\end{thm}

\begin{proof}
By Theorem~\ref{thm:stbound}, $2\leqslant s,t\leqslant n-2$ and $2\leqslant s+t \leqslant n$. If $n$ is even, then because $s,t$ are integers such that $2\leqslant s+t \leqslant n$, the largest value of $st=\lambda(n-1)$ is $\frac{n^2}{4}$, and so $\lambda\leqslant \frac{n^2}{4(n-1)}$. Since $\lambda$ is an integer, $\lambda\leqslant \frac{n}{4}$ if $n\equiv 0 \pmod{4}$ and $\lambda\leqslant \frac{n-2}{4}$ otherwise. If $n$ is odd, then the same argument yields $\frac{n^2-1}{4}$ as the maximum value of $st=\lambda(n-1)$. Thus, $\lambda\leqslant \frac{n^2-1}{4(n-1)}=\frac{n+1}{4}$. Since $\lambda$ is an integer, we have $\lambda\leqslant \frac{n+1}{4}$ if $n\equiv 3 \pmod{4}$ and $\lambda\leqslant \frac{n-1}{4}$ otherwise.
\end{proof}

In Theorem~\ref{thm:maxlambda},  the upper bound $\frac{n}{4}$ for $\lambda$ whenever $n\equiv 0 \pmod{4}$ is tight for the infinite family of Hamming schemes $H(n,2)$ and $H(n,4)$ (see Theorems~\ref{thm:Hamming,q=2} and~\ref{thm:Hamming,q=4}, respectively). Moreover, the upper bound $\frac{n-1}{4}$ for $\lambda$ whenever $n\equiv 1 \pmod{4}$ is tight for the infinite family of 2-class schemes corresponding to conference graphs (see Theorem~\ref{thm:SRGs}(3)), while the upper bound $\frac{n+1}{4}$ for $\lambda$ whenever $n\equiv 3 \pmod{4}$ is tight for all asymmetric 2-class schemes (see Theorem~\ref{thm:Asym2class}).

The following result in \cite[Theorem 1.21]{kreher2025lambda} also applies to association schemes, which improves the upper bound on $\lambda$ in Theorem~\ref{thm:maxlambda} depending on the factorization of $n-1$, $\lambda$, $s$ and $t$.

\begin{thm}
\label{thm:lambda}
Suppose $(X,\mathcal{A})$ is an association scheme on $n$ points that admits a nontrivial $\lambda$-fold $(s,t)$ near-factorization, where $n-1=n_1n_2$, $\lambda=\lambda_1\lambda_2$, $s=\lambda_1n_1$ and $t=\lambda_2n_2$. Then
\[
\lambda\leqslant 
  \begin{cases}
    \frac{n+1}{4}, & \text{if $n_1$ and $n_2$ are even},\\
    \frac{(n_1+1)(n-n_2+1)}{4n_1}, & \text{if $n_1$ is odd and $n_2$ is even},\\
    \frac{(n_2+1)(n-n_1+1)}{4n_2}, & \text{if $n_1$ is even and $n_2$ is odd},\\
    \min\big\{\frac{(n_1+1)(n-n_2+1)}{4n_1},\frac{(n_2+1)(n-n_1+1)}{4n_2}\big\}, & \text{if $n_1$ and $n_2$ are odd}.
  \end{cases}
\]
\end{thm}

If $n_1,n_2\geq 2$, then the assumption in Theorem~\ref{thm:lambda} implies $\operatorname{gcd}(s,n-1) >1$ and $\operatorname{gcd}(s,n-1) >1$, this also follows from Corollary~\ref{cor:n-1} and the fact that the near-factorization is nontrivial. 
Additionally, we remark that if $n_1,n_2\geq 2$, then the upper bounds for $\lambda$ when at least one of $n_1$ and $n_2$ is odd is always less than the upper bound $\frac{n+1}{4}$ when $n_1$ and $n_2$ are both even. This implies that $\lfloor\frac{n+1}{4}\rfloor$ is a general upper bound for $\lambda$ for any $n$ (this is also clear from Theorem~\ref{thm:maxlambda}).

We now derive improved upper bounds for $\lambda$ when $n\equiv k+1 \pmod{2k}$ for odd $k$.

\begin{cor}
\label{cor:lambdaimp}
Suppose $(X,\mathcal{A})$ is an association scheme on $n$ points that admits a nontrivial $\lambda$-fold $(s,t)$ near-factorization. If $n\equiv k+1 \pmod{2k}$ for some odd $k$, then $\lambda\leqslant \big\lfloor\frac{(k^2-1)n+(k+1)^2}{4k^2}\big\rfloor$.
\end{cor} 

\begin{proof}
Let $n\equiv k+1 \pmod{2k}$ for some odd $k$. Then we may write $n-1=2rk+k=p(2k+1)$ for some $r$. That is, $n-1=n_1n_2$ where $n_1=k$ and $n_2=\frac{n-1}{k}$ are both odd. Note that 
\begin{center}
$\displaystyle\frac{(n_1+1)(n-n_2+1)}{4n_1}=\frac{(k^2-1)n+(k+1)^2}{4k^2}\quad\text{and} \quad \frac{(n_2+1)(n-n_1+1)}{4n_2}=\frac{n^2-(p-1)^2}{4(n-1)}.$
\end{center}
Since $\frac{n^2-(k-1)^2}{4(n-1)}>\frac{n(k^2-1)+(k+1)^2}{4k^2}$ for all $k\geq 1$, applying Theorem~\ref{thm:lambda} yields the desired upper bound. 
\end{proof}

\subsection{Finding near-factorizations}

Theorems~\ref{thm:char1} and~\ref{thm:char2} gives two basic approaches in finding $\lambda$-fold $(s,t)$ near-factorization in association schemes. Assume $U,V \subseteq \{0,1,\dots,d\}$ are two subsets and set $S$ and $T$ as in Theorem~\ref{thm:char1} Condition~(1). 
 
The first approach involves the intersection algebra, one can check whether the corresponding sums of intersection matrices $B_U = \sum_{i \in U} B_i$ and $B_V = \sum_{j \in V} B_j$ satisfy 
\begin{equation}\label{eq:NFinIntAlg} 
B_U B_V = \lambda( B_1 + \dots + B_d ).
\end{equation}
This can be done by checking if the intersection numbers $\{p_{i,j}^k : 0 \leqslant i,j,k \leqslant d\}$ of the association scheme satisfy the conditions of Theorem~\ref{thm:char1}.

The second approach involves eigenvalues. Since the eigenvalues of $\lambda(J-I)$ are $\lambda(n-1)$ with multiplicity $1$, and $-\lambda$ with multiplicity $n-1$, we need to check if the valencies $s$ and $t$ of $S=S_U$ and $T=T_V$, respectively, will satisfy $s t = \lambda(n-1),$ while all other irreducible characters $\theta$ of $\mathcal{A}$ satisfy $\theta(S) \theta(T) = - \lambda$.    

In the remainder of the paper, we will use these approaches to find near-factorization in many different association schemes. 

\section{Near-factorizations in $2$-class association schemes}\label{sec:2class}

In this section, we give a characterization of near-factorizations in the smallest nontrivial association schemes. We have two cases: when the classes are symmetric (these correspond to \emph{strongly-regular graphs}), and when they are not symmetric (i.e. the two classes are each other's transpose).

\subsection{Symmetric $2$-class association schemes} 

A \emph{strongly-regular graph $\Gamma$ with parameters $(n, k; a, c)$} is a graph with $n$ vertices and valency $k$, where any two adjacent vertices have exactly $a$ common neighbours, and any two non-adjacent vertices have exactly $c$ common neighbours. If $\Gamma$ is a strongly-regular graph with parameters $(n,k; a,c)$, then its complement $\overline{\Gamma}$ is also strongly regular with parameters $(n, n-k-1; n-2k+c-2, n-2k+a)$.  Not all parameters $(n, k; a, c)$ are feasible. For example, a simple edge counting shows that the parameters $a, c$ must satisfy the compatibility condition $(n-k-1) c = (k-a-1)k$. A strongly-regular graph $\Gamma$ is \emph{primitive} if $\Gamma$ and its complement are both connected.  Primitive strongly-regular graphs correspond 
precisely to symmetric 2-class association schemes. For more on strongly-regular graphs see~\cite{BCN}.  

If $\Gamma$ is a strongly-regular graph with parameters $(n, k; a , c)$, then the a symmetric $2$-class association scheme of order $n$ are 
\begin{center}
$\big\{A_0=I, A_1 = A(\Gamma), A_2 = A(\overline{\Gamma})  \big\},$
\end{center}
\noindent where $A(\Gamma)$ and $A(\overline{\Gamma})$ are the adjacency matrices of $\Gamma$ and its complement, respectively. Note that $A_1^2 = k A_0 + a A_1 + c A_2$. Using this relation and the structure constant identities, which hold for all association schemes, we can completely determine the three intersection matrices for any strongly-regular graph $\Gamma$ (hence any symmetric $2$-class association) with parameters $(n, k ; a, c)$. These are:  
$$
B_0=I,
\quad 
B_1 = 
\begin{bmatrix}
 0 & k & 0 \\ 
 1 & a& k-a-1 \\ 
 0 & c & k-c 
\end{bmatrix}
\quad \textrm{and} \quad
B_2 = \begin{bmatrix} 
0 & 0 & \ell  \\ 
0 & k-a-1 & \ell - k + a+1 \\ 
1 & k - c & \ell - k + c 
\end{bmatrix},
$$ 
where $\ell = n-k-1$ since $A_2=A(\overline{\Gamma})$. As a consequence, $\ell c = (k-a-1)k$ and thus,
\begin{equation}
\label{eqsrgB}
B_1 B_2 = (k-a-1)B_1 + (k-c)B_2.
\end{equation}
The eigenvalues of a strongly-regular graph $\Gamma$ can be determined directly from the graph parameters. These are the eigenvalues of the intersection matrix $B_1$, i.e. the roots of
\begin{center}
$\det(xB_0-B_1)=(x-k)(x^2+(c-a)x-(k-c))$.
\end{center}
Clearly, the degree $k$ is an eigenvalue with multiplicity 1. Setting $\Delta = \sqrt{(c-a)^2+4(k-c)}$, the quadratic formula yields  the other distinct eigenvalues:
\begin{center}
$\theta  = \dfrac{ a-c +  \Delta}{2}  \quad\text{and}\quad \tau  = \dfrac{ a-c - \Delta }{2}. $
\end{center}
Consequently, there are precisely three distinct eigenvalues and they satisfy $k \geqslant \theta \geqslant 0 > -1 \geqslant \tau$.

There are only two cases of 2-class association schemes that admit near-factorizations.  We state them in terms of the parameters of the corresponding strongly-regular graph. Note that each case is stated twice, once for the graph $\Gamma$ in the association scheme and again for the other graph in the association scheme, $\overline{\Gamma}$.

\begin{thm}\label{thm:SRGs}
Let $\Gamma$ be a primitive strongly-regular graph with parameters $(n, k ; a, c)$. 
Then the symmetric $2$-class association scheme corresponding to $\Gamma$ admits a nontrivial $\lambda$-fold $(s,t)$ near-factorization if and only if one of the following conditions hold:
\begin{enumerate}[(i)]
\item  $c=a+ 2$, with $\lambda = k-a-1 = (1+ \theta)^2$,  and $(s, t ) = ( k+1, n-k-1)$; \label{SRGcase1}
\item  $c=a$, with $\lambda=k-a=\theta^2$, and
  $( s, t ) = ( k, n-k) $;
  \label{SRGcase2}
\item $c=a+1$, with $\lambda = k-a-1 =\theta^2+ \theta $ and $(s, t ) =( k, n-k-1)$.   \label{SRGcase3}
\end{enumerate}
\end{thm}

\begin{proof} 
Let $(X,\mathcal{A})$ be a symmetric $2$-class association scheme whose basis matrix $A_1$ is the adjacency matrix of a strongly-regular graph with parameters $(n, k ; a, c)$.  If $(X,\mathcal{A})$ admits a nontrivial near-factorization, then, assuming $A_1$ is involved in $S$, there are three cases. The first is when $S$ contains two classes and $T$ contains one class, the second is when $S$ contains one class and $T$ contains two classes, and the last one is when both $S$ and $T$ only have one class. 

\noindent {\bf Case (i):}  $S = A_0 + A_1$ and $T=A_2$. Translating the condition that $( A_0 + A_1) A_2 = \lambda( J-I)$ in terms of intersection matrices gives the identity 
$(B_0 + B_1)B_2 = \lambda(B_1 + B_2).$
Expanding the left-hand side of this equation using Equation (\ref{eqsrgB}) yields $B_2 + B_1B_2 = (k-a-1) B_1 + (k-c+1)B_2,$
which shows $\lambda= k-a-1=k-c+1$. Consequently,
such an  association scheme will admit a near-factorization with $\lambda = k-a-1$ whenever $c=a+2$.  In this case $(s, t )=(k+1, \ell)$ and, since $\theta = - 1 + \sqrt{(1+k-c)}$, we get $\lambda = (1+ \theta)^2$.

\noindent {\bf Case  (ii):} This is equivalent to Case (i), but with $\Gamma$ and $\overline{\Gamma}$ switched. Thus, the association schemes are the same. Using Equation (\ref{eqsrgB}), we obtain $B_1(B_0 + B_2) = (k-a) B_1 + (k-c) B_2.$ 
Therefore, Equation~(\ref{eq:NF}) has a solution when $\lambda=k-a=k-c$, and $( s, t ) =( k, 1+\ell)$. Since $a=c$ and $\theta = \sqrt{(k-c)}$, we get $\lambda = \theta^2$. 

\noindent {\bf Case  (iii):} $S = A_1$ and $T=A_2$. We can write this in terms of intersection matrices as $B_1 B_2 = \lambda(B_1 + B_2).$ In this case,
$B_1B_2 = (k-a-1)B_1 + (k-c)B_2$. Consequently, we need $\lambda=k-a-1=k-c$. Every $(n, k; a, c)$ with $c=a+1$ admits a solution with $\lambda=k-a-1=k-c$, and $(s,t ) =( k, \ell)$. Since $c=a+1$, we get 
$\theta = \frac{-1 + \sqrt{1 + 4(k-c)}}{2}$. As a result $(2 \theta +1)^2 = 1 + 4(k-c)$, and thus, $\lambda = k-c = \theta^2+ \theta$. 
\end{proof} 

Theorem~\ref{thm:SRGs}(\ref{SRGcase3}) includes the family of conference graphs. These are the strongly-regular graphs with parameters $( 4a+5,2 (a+1) ; a, a+1)$, where $a \geqslant 0$ with $\lambda =a+1$. In Appendix~\ref{app-SRG}, we give a table of the strongly-regular graphs of order up to $100$ that admit a near-factorization.
  
Apart from conference graphs, we find one more infinite family of examples in~\cite[page 386]{BCN}, provided by certain collinearity graphs of generalized quadrangles. The generalized quadrangle $GQ(u, v)$ is a strongly-regular graph with parameters $( (u+1)(uv+1), \, u(v+1) ; u-1, \, v+1)$. In this case, $\theta = u-1$. For $u, v \geqslant 2$, these are known to exist whenever 
\begin{center}
$(u,v) \in \left\{ (u,1), \, (1,v), \, (q,q), \, (q,q^2), \, (q^2,q), \, (q^2,q^3), \, (q-1,q+1), \,(r+1,r-1) \right \},$
\end{center} 
where $q$ a prime power or $r$ a power of $2$. Apart from some small examples, $GQ(2^m+1,2^m-1)$ satisfies Case~(\ref{SRGcase1}) with $\lambda=2^{2m}$, while $GQ(q, q)$ satisfies Case~(\ref{SRGcase2}) with $\lambda=q^2$ for any prime power $q$.

\subsection{Asymmetric $2$-class association schemes} 

We now investigate near-factorizations in asymmetric $2$-class association schemes, whose set of adjacency matrices is $\{A_0=I, A_1,  A_1^{\top}\}$. The intersection numbers of these association schemes only depend on the order $n$, which must be congruent to $3$ modulo $4$.

\begin{thm}\label{thm:Asym2class}
Let $(X,\mathcal{A})$ be an asymmetric $2$-class association scheme of order $n=4u+3\geqslant 7$.  Then $(X,\mathcal{A})$ admits the nontrivial non-symmetric $(u+1)$-fold $(2u+1, 2u+2)$ near-factorizations: 
\begin{equation*}
A_1(A_0+A_1) = (u+1)(J-I)\quad \text{and}\quad A_1^\top(A_0+A_1^\top) = (u+1)(J-I). 
\end{equation*}
\end{thm} 

\begin{proof} 
The adjacency matrices of an asymmetric $2$-class scheme with order $n = 4u+3$ satisfies $A_1^2 = uA_1 + (u+1)A_1^{\top}$ and $(A_1^{\top})^2 = uA_1^{\top} + (u+1)A_1.$ As $A_1+A_1^{\top}=J-I$, adding $A_1$ and $A_1^\top$ respectively to the preceding identities, and then simplifying gives the desired near-factorizations.
\end{proof}

Any cyclic group of prime order $p=4u+3$ decomposes as the set of non-identity quadratic residues, and the set of non-identity quadratic non-residues.  Since $-1$ is a non-residue, this decomposition produces an asymmetric $2$-class fusion of the thin scheme of this group, which, therefore, has a $(u+1)$-fold $(2u+1,2u+2)$ near-factorization indicated in Theorem~\ref{thm:Asym2class}.

\subsection{Half-set near-factorizations}
\label{subsec:halflings}

Several examples of strongly-regular graphs with near-factorizations have an additional property that allows us to construct larger association schemes that also have a near-factorization. A \emph{half-set near-factorization} is a near-factorization with parameters satisfying $s = 2\lambda$ and $t= \frac{n-1}{2}$. By~\cite[Theorem 2.19]{kreher2025lambda}, half-set near-factorizations in groups are preserved by direct products. The analogous statement holds for the tensor product of association schemes.

\begin{thm}
\label{thm:halfsetProds}
Assume for $i =1,2$, that $\mathcal{A}_i$ is an association schemes of order $n_i$ admitting a $\lambda_i$-fold $(s_i, t_i)$ near-factorization. Then the association scheme $\mathcal{A}_1 \otimes \mathcal{A}_2$ admits a nontrivial half-set $2\lambda_1 \lambda_2$-fold $( s_1 s_2, \frac{n_1 n_2 - 1}{2} )$ near-factorization. 
\end{thm}

\begin{proof}
Assume that $S_1 T_1 = \mu_1(J_1-I_1)$ is a nontrivial half-set near-factorization of $\mathcal{A}_1$ with parameters 
$(s_1, t_1) = (2\lambda_1,   \frac{n_1-1}{2} )$. Further, let $S_2 T_2 = \lambda_2(J_2-I_2)$ be a nontrivial half-set near-factorization of $\mathcal{A}_2$ is  with parameters $(s_2, t_2) = ( 2\lambda_2, \frac{n_2-1}{2})$.
It is straightforward to show that 
\begin{center}
$ (S_1 \otimes S_2) \left(  ( T_1 \otimes (J_{n_2}-T_2) ) +  (  (J_{n_1}-T_1) \otimes T_2 ) \right)  = 
      2 \lambda_1 \lambda_2 ( ( J_{n_1} \otimes J_{n_2} ) -  ( I_{n_1} \otimes I_{n_2})  ).$
\end{center}
Thus $S_1 \otimes S_2$ and $( T_1 \otimes (J_{n_2}-T_2) )  + (  (J_{n_1}-T_1) \otimes T_2 )$ is a nontrivial half-set near-factorization of $\mathcal{A}_1 \otimes \mathcal{A}_2$ with $(s, t) = ( s_1 s_2, \frac{n_1 n_2 - 1}{2} )$ and $\lambda =  2 \lambda_1 \lambda_2$.  
\end{proof}

In Case~(\ref{SRGcase1}) of Theorem~\ref{thm:SRGs}, the strongly-regular graphs with parameters $(n, k ; \mu,\mu+2)$ admit nontrivial half-set near-factorizations of the form $(I+A_1)A_2=\lambda(J-I)$ with parameters $s=k+1$, $t=2\lambda = k+2$.
Several examples of such graphs are recorded in Table~\ref{Table:SRGcase1} given in Appendix~\ref{app-SRG}. The strongly-regular graphs in Case~(\ref{SRGcase3}) of Theorem~\ref{thm:SRGs} always lead to nontrivial half-set near-factorizations of the form $A_1A_2=\lambda(J-I)$ and parameters $s=k = 2 \lambda$ and $t=k$. 


\subsection{Paley graphs and tournaments}
\label{subsec:Paley}

We introduce an infinite family of strongly-regular graphs that all have nontrivial half-set near-factorizations.
Given a finite field $\mathds{F}_q$, its set of \emph{quadratic residues} is the set $QR=\{x^2 \mid x \in \mathds{F}_q\}$ and the set of \emph{non-quadratic residue} is $NR=\mathds{F}_q-QR$. There are $\frac{q-1}{2}$ non-zero quadratic residues and thus, there are $\frac{q-1}{2}$ non-quadratic residues. If $q \equiv 1 \pmod{4}$, then the \emph{Paley graph of order $q$}, denoted $P(q)$, is the Cayley graph of $\mathds{F}_q$ with generating set $QR$. The Paley graph $P(q)$ is a strongly-regular graph with parameters $(q, \frac{q-1}{2} ; \frac{q-5}{4}, \frac{q-1}{4})$ that is also self-complimentary. An analogous definition applied to $q \equiv 3 \pmod{4}$ gives rise to the \emph{Paley tournament of order $q$}, denoted $T(q)$. The graph $T(q)$ is a directed graph by virtue of the fact that $-1$ is not a quadratic residue. Since $T(q)$ is a directed graph, its adjacency matrix is not symmetric.  

If $q \equiv 1 \pmod{4}$ then the Paley graph $P(q)$ falls under Case~(\ref{SRGcase3}) of Theorem~\ref{thm:SRGs}, so it admits a nontrivial half-set near-factorization with $A_1$ and $A_2$ with $\lambda = \frac{q-1}{4}$. By Theorem~\ref{thm:Asym2class}, if $q \equiv 3 \pmod{4}$ then the Paley tournament $T(q)$  has a nontrivial half-set near-factorization with $I+A_1$ and $A_2$. These lead to product association schemes that have nontrivial half-set near-factorizations, which are discussed in Subsection~\ref{subsec:schurianProduct}.

\section{Generalized hexagons}
\label{sec:genhex}

In this short section, we establish the existence of $\lambda$-fold near-factorizations in a family of $P$-polynomial association schemes that arise from a particular distance-regular graph. A connected graph of diameter $D$ is \emph{distance regular} if there exists constants $c_i, a_i$, and $b_i$ such that for all $i\in \{0,1,\ldots, D\}$ and all pairs of vertices $x$ and $y$ at distance $i$, amongst the neighbours of $y$ there are $c_i$ vertices at distance $i-1$ from $x$, $a_i$ vertices at distance $i$ from $x$, and $b_i$ vertices at distance $i+1$ from $x$. The \emph{intersection array} of a distance-regular graph records the values of $b_i$ for $i \in \{0,1,\ldots, D-1\}$, followed by the values of $c_i$ for $i \in \{1,\ldots, D\}$.  
The \emph{generalized hexagon of order $(q, q)$} for $q>1$, is the distance-regular graph of order $\frac{q^6-1}{q-1}$ with intersection array 
$$[q^2+q, q^2, q^2 ; 1,1,q+1].$$ 

Generalized hexagons are known to exists when $q$ is a prime power. This is a distance-regular graph of diameter four and its distance matrices form a four class $P$-polynomial association scheme. Using the intersection array, we derive the intersection matrices of the corresponding association schemes as follows:
$$ 
B_0=I,\quad B_1 = \begin{bmatrix} 
0 & q^2+q & 0 & 0 \\ 
1 & q-1 & q^2 & 0 \\ 
0 & 1 & q-1 & q^2 \\ 
0 & 0 & q+1 & q^2-1 
\end{bmatrix}, 
\quad
B_2 = \begin{bmatrix} 
0 & 0 & q^4+q^3 & 0 \\ 
0 & q^2 & q^3-q^2 & q^4 \\ 
1 & q-1 & q^3+q^2-q & q^4-q^2 \\ 
0 & q+1 & q^3+q^2-q-1 & q^4-q^2 
\end{bmatrix},
$$ 
and 
$$ B_3 = \begin{bmatrix} 
0 & 0 & 0 & q^5 \\ 
0 & 0 & q^4 & q^5-q^4 \\ 
0 & q^2 & q^4-q^2 & q^5-q^4 \\ 
1 & q^2-1 & q^4-q^2 & q^5-q^4 
\end{bmatrix}. 
$$
The scheme that gives rise to these intersection array admits a $\lambda$-fold near-factorization resulting from the identity $B_3 (B_0+B_1+B_2) = q^4(B_1+B_2+B_3),$
with parameters $\lambda=q^4$ and $(s, t) = (q^5, q^4+q^3+q^2+q+1)$.


\section{Near-factorizations in Hamming association schemes}
\label{sec:Hamming}

In this section, we investigate the question of existence for near-factorizations in Hamming schemes which we now define. For positive integers $n$ and $q$, we denote the set of words of length $n$ over alphabet $Q$, where $|Q|=q$, by $Q^n$. The \emph{Hamming scheme} $H(n, q) = \{A_0=I, A_1, \dots, A_n\}$ is an association scheme on point set $Q^n$. In this association scheme, two words are $i$-related if they differ in exactly $i$ coordinates. The eigenvalues of $A_i$ in the Hamming scheme $H(n, q)$ are given by the values of \emph{Krawtchouk polynomials}. This polynomial implies an ordering on the eigenspaces, called the \emph{standard ordering}, that we use throughout this section (indeed, the Hamming scheme is both P-polynomial and Q-polynomial under the standard ordering). See~\cite[Chapter 10]{MR3497070} for more details on the Hamming association scheme.

\begin{thm}\label{thm:Kraw}
The $j^{th}$ eigenvalue of $A_i$ is
\[ 
	p_i(j)=\sum_{r=0}^j (-1)^r (q-1)^{i-r}\binom{n-j}{i-r}\binom{j}{r}
\]
with multiplicity $(q-1)^j\binom{n}{j}$ for $j=0, \dots, n$.
\end{thm}

We now introduce notation that will be used throughout this section as well as Sections~\ref{sec:Johnson} and \ref{sec:grassman}. Given a subset of the relation $U \subseteq \{0,1,\dots,n\}$, we let $S_U = \sum_{i \in U} A_i$. 
Since $S_U$ belongs to the adjacency algebra of the Hamming scheme, its eigenvalues are
$$
	\mu_j( S_U ) = \sum_{i \in U} p_i(j).
$$ 
Using Theorem~\ref{thm:char2}, we will use the eigenvalues of the Hamming scheme to determine when a near-factorization exists.

We organize this section as follows. First, we establish the existence of $2^{n-2}$-fold near-factorizations of $H(n,2)$ when $n$ is odd. Then, we show that $H(n, q)$ admits a nontrivial $\lambda$-fold near-factorizations for $q \in \{3,4\}$ and all $n\geqslant 2$. We conclude by showing that there does not exists a nontrivial $\lambda$-fold near-factorization of $H(n,q)$ when $q >n$. The  $\lambda$-fold near-factorizations found in this section were previously known to exist, as they arise from $\lambda$-fold near-factorizations of the abelian group $\mathbb{Z}^n_q$ as shown in~\cite{kreher2025lambda}. Using different proof techniques, we show that these $\lambda$-fold near-factorizations of $\mathbb{Z}^n_q$ occur as fusion of classes of the Hamming scheme.  

\subsection{Hamming schemes with $q=2$}   
\label{subsec:Hamming2}

In this subsection, we show that $H(n,2)$ admits a nontrivial $2^{n-2}$-fold near-factorization realized by partitioning the relations of the scheme according to their congruency modulo 4. Namely, we will show that
\begin{center}
$S_U  T_V = 2^{n-2} ( J - I )$,
\end{center}
where
$$ 
U=\{~i \mid i \equiv 0,3 \pmod{4}\}\quad \text{and} \quad V=\{~i \mid i \equiv 1,2 \pmod{4}\}. 
$$ 
The two corresponding graphs are
$$
	S_U = \sum_{i \in U } A_i,  \qquad T_V = \sum_{j \in V} A_j.
$$

In order to show that $(S_U, T_V)$ is a near-factorization, we will compute the spectra of $S_U$ and $T_V$. By substituting $q=2$ into the formula of Theorem~\ref{thm:Kraw}, we obtain the eigenvalue of $A_i$ realized by the $j^{th}$ eigenspace
\begin{equation}\label{equ:kratou}
	p_i(j)=\sum_{r=0}^j (-1)^r  \binom{n-j}{i-r} \binom{j}{r}
\end{equation}
with multiplicity $\binom{n}{j}$. 
To make use of Equation~(\ref{equ:kratou}) to derive the spectra of $S_U$ and $T_V$, as done in Lemmas~\ref{lem:val} and~\ref{lem:jeig}, we will require several technical lemmas, namely Lemmas~\ref{lem:sums}-\ref{lem:jeig}, on certain sums of binomial coefficients. Central to our computations are the two well-known binomial identities given below.

\begin{lem}\label{lem:sum0} Let $n$ be a positive integer. If $r\in \{0,2\}$, then
$$
\sum_{i \equiv r \pmod{4}} \binom{n}{i} 
=2^{n-2}+(-1)^{\frac{r}{2}}(2^{\frac{n}{2}-1})\cos\bigg(\frac{n\pi}{4}\bigg).
$$
If $r \in \{1,3\}$, then
$$
\sum_{i ~ \equiv ~ r \pmod{4}} \binom{n}{i} 
=2^{n-2}+(-1)^{\frac{r-1}{2}}(2^{\frac{n}{2}-1})\sin\bigg(\frac{n\pi}{4}\bigg).
$$
\end{lem}

We use Lemma~\ref{lem:sum0} to obtain Lemmas~\ref{lem:sums} and~\ref{lem:sums2}. The proof of Lemma~\ref{lem:sums} can be found in Appendix~\ref{app:calc}.

\begin{restatable}{lem}{sums}
\label{lem:sums}
Let $n$ be even. Then
$$ 
\sum_{i~\equiv ~ 1,2 \pmod{4}} \binom{n}{i}
-\sum_{i~\equiv 0, 3 \pmod{4}} \binom{n}{i}
 = 
 \begin{cases}
  2^{\frac{n}{2}}   &  \textrm{if}~~ n \equiv 2,4 \pmod{8}, \\
  -2^{\frac{n}{2}}&  \textrm{if}~~ n \equiv 0,6 \pmod{8}. \\
\end{cases} 
$$
\end{restatable}

In Lemma~\ref{lem:sums2} below, we derive two identities stated relative to $j$, the index of the eigenspaces. We omit this proof since it is very similar to Lemma~\ref{lem:sums}.

\begin{lem} \label{lem:sums2}
Let $j$ be a positive integer such that $0 \leqslant i \leqslant j$. Then
$$ 
\sum_{i \equiv 0,1 \pmod{4}} \binom{j}{i}  -  \sum_{i \equiv 2,3 \pmod{4}} \binom{j}{i}
= 
\begin{cases}
2^{\frac{j}{2}} & j \equiv 0, 2 \pmod{8}, \\
-2^{\frac{j}{2}} &  j \equiv 4,6\pmod{8},\\
2^{\frac{j+1}{2}} &  j \equiv 1 \pmod{8},\\
-2^{\frac{j+1}{2}} &  j \equiv 5 \pmod{8},\\
0 &  j \equiv 3, 7 \pmod{8},\\
\end{cases}
$$
and
$$ 
\sum_{i~\equiv ~ 0,2~\pmod{4}} \binom{j}{i}-\sum_{i~\equiv~1,3~\pmod{4}} \binom{j}{i}=0.
$$
\end{lem}

We now use Lemma~\ref{lem:sum0} to compute $\mu_0(S_U)$ and $\mu_0(T_V)$. 

\begin{lem} \label{lem:val}
Let $n$ be an even integer. If $U= \{i \mid i\equiv  0,3 \pmod{4} \}$, and $S_U=\sum_{i \in U } A_i$, then the valency of $S_U$ is 
$$
\mu_0(S_U)= \sum_{i~ \equiv~ 0,3 \pmod{4}} \binom{n}{i}
=
\begin{cases}
  2^{n-1}  +  {2}^{\frac{n}{2} -1 }   &  \textrm{if}~ n \equiv 0,6 \pmod{8}, \\
  2^{n-1}  -  2^{\frac{n}{2} - 1} &  \textrm{if}~ n \equiv 2,4 \pmod{8}. \\
\end{cases} 
$$
If $V=\{i \mid i\equiv  1,2 \pmod{4} \}$ and $T_V = \sum_{j \in V} A_j$, then the valency of $T_V$ is 
$$
\mu_0(T_V) = \sum_{j~\equiv~ 1,2\pmod{4}} \binom{n}{j}
=
\begin{cases}
  2^{n-1}  -  {2}^{\frac{n}{2} -1 }   &  \textrm{if}~ n \equiv 0,6 \pmod{8}, \\
  2^{n-1}  +  2^{ \frac{n}{2} - 1} &  \textrm{if}~ n \equiv 2, 4 \pmod{8}. \\
\end{cases} 
$$
\end{lem}

As a corollary to Lemma~\ref{lem:val}, we have the following expression for $\mu_0(S_U T_V)$. 

\begin{cor}
    
\label{cor:valXY,q=2}
For all even $n$, $\mu_0(S_U T_U)=2^{n-2} (2^n-1).$
\end{cor}

\begin{proof}
In all cases
$\mu_0(S_U) \mu_0(T_V) = (2^{n-1}  +  {2}^{ \frac{n}{2} -1 } ) ( 2^{n-1}  -  {2}^{  \frac{n}{2} -1 }  ) = 2^{n-2} (2^n-1)$.
\end{proof}

To show that $S_U$ and $T_V$ form a $2^{n-2}$-fold near-factorization of $H(n,2)$, it suffices to show that the spectrum of $S_U T_V$ is equal to the spectrum of $2^{n-2}(J-I)$. Thus, we we now compute the remaining eigenvalues of $S_U$ and $T_V$. If $j\neq 0$, then Equation~(\ref{equ:kratou}) yields

\begin{equation}
\label{equ:eigS}
\mu_j(S_U)=\sum_{i \equiv 0,3 \pmod{4}} p_i(j) =  \sum_{i  \equiv 0,3\pmod{4}} \mbox{ } \sum_{r=0}^j  (-1)^r \binom{n-j}{i-r} \binom{j}{r};
\end{equation}

\begin{equation}
\label{equ:eigT}
\mu_j(T_V)=\sum_{i \equiv 1,2  \pmod{4}} p_i(j) =  \sum_{i \equiv 1,2\pmod{4}} \mbox{ } \sum_{r=0}^j  (-1)^r \binom{n-j}{i-r} \binom{j}{r}.
\end{equation}

We simplify Equations~(\ref{equ:eigS}) and~(\ref{equ:eigT}) in the following lemma (the proofs can be found in Appendix~\ref{app:calc}).

\begin{restatable}{lem}{jeig}
\label{lem:jeig}
Let $U= \{i \mid i \equiv  0,3 \pmod{4} \}$ and $V= \{i \mid i\equiv  1,2 \pmod{4} \}$. 
If $S_U=\sum_{i \in U } A_i$ and $T_V=\sum_{i \in V } A_i$, then $\mu_j(S_U)= \pm 2^{n/2-1}$ and $\mu_j(T_V) =\mp 2 ^{\frac{n}{2}-1}$ for all $j>0$.
\end{restatable}

\begin{thm}\label{thm:Hamming,q=2}
If $n$ is even, then $H(n,2)$ has a nontrivial $2^{n-2}$-fold $(2^{n-1}-2^{n/2 -1},2^{n-1}+2^{n/2 -1})$ near-factorization realized by $S_U=\sum_{i \in U } A_i$ and $T_V=\sum_{i \in V } A_i$ where 
$U=\{i \mid i \equiv 0,3 \pmod{4} \}$ and 
$V=\{i \mid i \equiv 1,2 \pmod{4} \}.$
\end{thm}

\begin{proof} 
By Corollary~\ref{cor:valXY,q=2}, the valency of $S_U T_V$ is $2^{n-2}(2^n-1)$.
If $j \neq 0$, then 
$$ 
\bigg(\sum_{i~ \equiv 0,3 \pmod{4}} p_i(j) \bigg) \bigg(\sum_{i ~ \equiv 1,2 \pmod{4}} p_i(j)\bigg)=-2^{n-2}.
$$ 
Therefore, $-2^{n-2}$ is an eigenvalue of $S_U T_V$ with multiplicity $2^n-1$.  Thus, when $n$ is even, $S_U$ and $T_V$ is a $2^{n-2}$-fold near-factorization in $H(n,2)$.  
\end{proof}

By a similar reasoning, it can be shown that the pair $(S_U, T_V)$ defined below is also a $2^{n-2}$-fold $(2^{n-1}-2^{n/2 -1},2^{n-1}+2^{n/2 -1})$ near-factorization of $H(n,2)$ when $n$ is even:
\begin{equation*}
U= \{i \mid i\equiv  2,3 \pmod{4} \} \quad \textrm{and} \quad V=\{i \mid i\equiv  0,1 \pmod{4} \}.
\end{equation*}

When $n$ is odd, we verified that no $\lambda$-fold near-factorizations of $H(n,2)$ exists whenever $n \leqslant 15$. 
However, we were unable to prove non-existence for all odd $n$, and so we make the following conjecture. 

\begin{conj}
For $n$ odd and $\lambda \geqslant 1$, $H(n,2)$ does not admit any near-factorizations.
\end{conj}

In Section~\ref{subsec:HalvedCube}, we give a nontrivial near-factorization of a scheme derived from $H(n,2)$ when $n$ is odd.


\subsection{Hamming schemes with $q \in \{3, 4\}$}

Next, we establish the existence of a near-factorization in $H(n,3)$ and $H(n,4)$. Before we do so, we prove two lemmas that apply to all $q$. The purpose of these two lemmas is to derive simple formulae for the eigenvalues of matrices that arise as the sum of matrices indexed by relations of the same parity. The proof of the these two lemmas can be found in Appendix~\ref{app:calc}.

In this subsection, we let $S_{\textit{even}}=\sum_{i \ \textrm{even}} A_i$ and $S_{\textit{odd}}=\sum_{i \ \textrm{odd}} A_i$. 

\begin{restatable}{lem}{valXY}
\label{lem:valXY}
If $n$ is even, 
$$
\mu_0(S_{\textit{even}}) = \frac{1}{2} ( q^n + (q-2)^n); \quad 
\mu_0(S_{\textit{odd}}) =  \frac{1}{2} ( q^n - (q-2)^n).
$$ 
If $n$ is odd, 
$$
\mu_0(S_{\textit{even}}) =  \frac{1}{2} (q^n - (q-2)^n); \quad 
\mu_0(S_{\textit{odd}}) =  \frac{1}{2} (q^n + (q-2)^n).
$$ 
\end{restatable}

\begin{restatable}{lem}{otherXYeigs}
\label{lem:otherXYeigs}
 If  $j \in \{1,2, \ldots, n\}$, then 
 $$
 \mu_j(S_{\textit{even}}) =  (-1)^{n-j}  2^{j-1} (q-2)^{n-j}; \quad 
\mu_j(S_{\textit{odd}}) =  (-1)^{n-j+1} 2^{j-1} (q-2)^{n-j}.
$$
\end{restatable}

We now consider the case $q=3$ for which we establish the existence of a $2^{n-1}$-fold near-factorization of $H(n, 3)$ for all $n \geqslant 3$. This near-factorization is realized by the pair of subsets of relations: $S= S_{\textit{even}}$ and $T= A_n$ when $n$ is odd, and $S= S_{\textit{odd}}$ and $T= A_n$ when $n$ is even. Similar to the $q=2$ case, we will use the spectra of $S$ and $T$ to prove that these form a near-factorization. Lemmas~\ref{lem:valXY} and \ref{lem:otherXYeigs} give the spectrum of $S$. Since $T=A_n$, we have $\mu_j(T)=p_j(n)$ from Theorem \ref{thm:Kraw}. 

\begin{lem}\label{lem:eigsXn}
Let $j \in \{0,1, \ldots, n\}$, then $\mu_j(T)=(-1)^j 2^{n-j}$.  
\end{lem}

\begin{thm}\label{thm:Hamming,q=3}
The Hamming scheme $H(n,3)$, with $n \ge 2$, admits a nontrivial $2^{n-1}$-fold $(2^n, \frac{3^n-1}{2})$ near-factorization with
$$
S = \begin{cases} 
 S_{\textit{even}}  & \mbox{ if $n$ is odd}, \\ 
 S_{\textit{odd}}& \mbox{ if $n$ is even,} 
\end{cases} 
\quad \mbox{ and }\quad 
T=A_n.
$$  
\end{thm}

\begin{proof}
First, assume $n$ is odd. Since $q=3$, the valency of $\mu_0(S)=  \frac{1}{2} ( 3^n - 1)$ and $\mu_0(T)=2^n$. Thus,  $\mu_0(S T)= 2^{n-1} (3^n-1)$. 
If $j\neq 0$, then $\mu_j(S)=(-1)^{n-j} 2^{j-1} $ (by Lemma~\ref{lem:eigsXn})  and 
$\mu_j(T)=(-1)^j 2^{n-j}$ (by Lemma~\ref{lem:otherXYeigs}). Therefore, $\mu_j(S T)=(-1)^n 2^{n-1}$. 
Thus $(S,T)$ gives the desired near-factorization of $H(n,3)$. A similar reasoning applies to the case $n$ even.
 \end{proof}

Note that the near-factorization in Theorem~\ref{thm:Hamming,q=3} is also a half-set near-factorization.


Applying a similar reasoning for $q=4$, we have the following result. 

\begin{thm}\label{thm:Hamming,q=4}
For all $n \geqslant 2$, $H(n,4)$ has a $4^{n-1}$-fold $(2^{n-1}(2^n-1), 2^{n-1}(2^n+1))$ near-factorization realized by 
$S=S_{\textit{even}}$ and $T=S_{\textit{odd}}$.
\end{thm}

\begin{proof} 
If $q=4$, then $\mu_j(T) = (-1)^{n-j} 2^{n-1}$ and $\mu_j(S) = (-1)^{n-j+1}2^{n-1}$ for all $j>0$. 
Hence, $\mu_j ( S T )  = - 2^{2n-2} = -4^{n-1}$ for any $j > 0$.  
By Lemma~\ref{lem:valXY}, $\mu_0(S) \mu_{T_U}=4^{n-1}(4^n-1)$.
\end{proof} 

\subsection{Hamming schemes with $q>n$} 
\label{SS:qnHam}

We now rule out the existence of $\lambda$-fold near-factorizations in the Hamming scheme when $q>n$ and $n\geqslant 5$. To prove that no nontrivial near-factorization exists, we show that when $q>n$, $|\mu_1(S_U T_V)|>|\mu_n(S_U T_U)|$ for all pairs of subsets $U, V \subset \{0,1,\dots, n\}$. By Condition~\ref{char2:cond3} of Theorem~\ref{thm:char2}, the desired contradiction will follow.  

Since we must consider all pairs of disjoint subsets of relations, we introduce the following definition. A pair $(U, V)$ of subsets of $\{0,1,\dots, n\}$ is called \emph{admissible} if $U \cap V=\varnothing$ and $U, V \neq \{0\}$. As a consequence of the last condition, if $(U,V)$ is admissible, then $S_U, T_V$ does not correspond to a trivial near-factorization. 

The eigenvalue of $A_i$ in $H(n, q)$ realized by the $1^{st}$-eigenspace can be computed using Theorem~\ref{thm:Kraw}, which in this case is the following two-term sum:
\begin{equation} 
\label{eq:Hammingeigenv1}
p_i(1)=(q-1)^i  \binom{n-1}{i} - (q-1)^{i-1} \binom{n-1}{i-1}
          =(q-1)^{i-1} \binom{n-1}{i-1} \left( \frac{n-i}{i}(q-1)-1 \right).
\end{equation}
The eigenvalue of $A_i$ realized by the $n^{th}$-eigenspace is
\begin{equation}
\label{eq:eigenn}
p_i(n)=(-1)^i\binom{n}{i}.
\end{equation} 
We establish an upper bound on $| \mu_n(S_U T_V) |$ over all admissible pairs $(U,V)$ of $\{0,1,\dots, n\}$. 

\begin{lem}
\label{lem:maxn}
Let $P$ be the set of all admissible pairs of $\{0,1,\dots, n\}$. Then
$$
	\max \left( \{|\mu_n(S_U T_V)| \, \mid \, (U,V) \in P\} \right) = 2^{2(n-1)}.
$$
\end{lem}

\begin{proof}
Note that 
$\max\{ | \mu_n(S_U T_V) |  \mid (U,V) \in P \}$ is achieved by $(U, V)$, where 
 $U=\{i \mid p_i(n)<0 \}$ and $V=\{i \mid p_i(n) \geqslant 0\}$. As per Equation~(\ref{eq:eigenn}), the maximum is achieved by the pair 
 $U=\{i \mid \textrm{$i$ is odd}\}$ and $V=\{i \mid \textrm{$i$ is even}\}$. Lemma~\ref{lem:sum0} then implies that 
 $| \mu_n(S_U) |=2^{n-1}=\mu_n( T_V )$ and the claim follows. 
\end{proof}

Our next objective is to derive a lower bound on $|\mu_1(S_U)|$ assuming that $\mu_1(S_U)<0$. To do so, we will exploit a sign pattern for the $1^{st}$-eigenspace that was first established in \cite{MR3862932}. 

\begin{prop}
Let $q \geqslant 2$ and $0 \leqslant i \leqslant n$. Then $p_i(1) <0$ if and only if $i \geqslant n-\frac{n-1}{q}$. 
\end{prop}

Clearly, if $q>n$, then $\frac{n-1}{q}<1$ and thus, we have the following corollary. 

\begin{cor}
If $q>n$, then $ p_i(1)>0$ if and only if $i=n$.  
\end{cor}

Thus, when $q>n$, there is precisely one negative eigenvalue on the $1^{st}$-eigenspace, which is realized by $A_n$. If $(U,V)$ gives rise to a $\lambda$-fold near-factorization of $H(n, q)$, then $\mu_1(S_U T_V)<0$. Therefore, $n$ must appear in one of $U$ or $V$. Without loss of generality, we will henceforth assume that $n \in U$ and thus, that $\mu_1(S_U)<0$. Lastly, by axiom (2) of Definition \ref{defn:AS}, $\sum_i p_i(1)=0$. Consequently, if $i \not \in U$, then $| \mu_1(S_U) | \geqslant p_i(1)$. 

\begin{cor}
\label{cor:lower1}
 Let $(U,V)$ be an admissible pair of $\{0,1,\dots, n\}$ such that $n \in U$.
  If $| \mu_1(S_U) | \leqslant p_i(1)$, then $i \in U$. 
\end{cor}

To show that $|\mu_1(S_U T_V)|>|\mu_n(S_U T_V)|$ for all admissible pairs $(U,V)$, we will also require the following very rough lower bound on certain eigenvalues realized by the $1^{st}$-eigenspace. 

\begin{lem}
\label{lem:lower1}
If $q>n$, $q \geqslant 17$, and $i \geqslant \frac{n}{2}$, then $p_i(1) > 2^{2(n-1)}$. 
\end{lem}

\begin{proof}
If $i>\frac{n+1}{2}$, then, under the given hypothesis, we have
$$
p_i(1) =(q-1)^{i-1} \binom{n-1}{i-1} \left( \frac{n-i}{i}(q-1)-1 \right)
	\geqslant
	(16)^{\frac{n+1}{2}-1}>  2^{2n+1};
$$
if $i = \frac{n}{2}$, then
$$
p_i(1) =(q-1)^{i-1} \binom{n-1}{i-1} \left( \frac{n-i}{i}(q-1)-1 \right)
	\geqslant
	(16)^{\frac{n}{2}-1}(15)
 >  2^{2n-1},
$$
Consequently, if $ \frac{n}{2} \leqslant i < n$, then $|p_i(1) |>2^{2n-2}$. 
\end{proof}

If $(U,V)$ is a $\lambda$-fold near-factorizations, then by Theorem \ref{thm:char2}, $\mu_1(S_U T_V) = \mu_n(S_U T_V)$. The following corollary is an immediate consequence of Corollary~\ref{cor:lower1} and Lemmas~\ref{lem:maxn} and~\ref{lem:lower1}.

\begin{cor}
\label{cor:set}
Let $q>n$, $q\geqslant 17$ and $(U, V)$ be an admissible pair of $\{0,1,\dots, n \}$ which gives rise to a $\lambda$-fold near-factorization of $H(n, q)$ with $\mu_1(S_U) <0$. If $i \geqslant \frac{n}{2}$, then $i \in U$. 
\end{cor}

We will also require the following lower bound on some of the eigenvalues realized by the $1^{st}$-eigenspace. 

\begin{lem}
\label{lem:n2}
If $q >n$ and $1\leqslant i \leqslant  \frac{n}{2}$, then $p_i(1)\geqslant (n-2)^2$.
\end{lem}

\begin{proof}
If $i=1$, then by Equation~(\ref{eq:Hammingeigenv1}), we have
\begin{align*}
p_i(1)= \left((n-1)(q-1)-1 \right)  \geqslant (n-2)^2.
\end{align*}
The result clearly holds for $n \leqslant 3$. Then, for $n \geqslant 4$, we have
\begin{align*}
p_i(1)=(q-1)^{i-1} \binom{n-1}{i-1} \left( \frac{n-i}{i}(q-1)-1 \right)
\geqslant (n-1)^2 \left( \frac{n-i}{i}(n)-1 \right) \geqslant  (n-2)^2.\ 
\end{align*}
\end{proof}

The next step is to show that, given an admissible pair $(U,V)$ with $\{ i \mid i\geqslant r\} \subseteq U$ and $r \leqslant \frac{n}{2}$, the eigenvalue of $S_U T_V$ on the $n^{th}$-eigenspace is strictly smaller than $p_r(1)$. 

In the proof of Proposition~\ref{prop:uppern}, we will require the following two standard binomial identities:

\begin{equation}
\label{eq:partial}
\sum_{i\leqslant r} (-1)^i\binom{n}{i}=(-1)^r\binom{n-1}{r};
\quad 
\binom{n}{r}=\frac{n}{r}\binom{n-1}{r-1}.
\end{equation}

\begin{prop}
\label{prop:uppern}
Assume $q>n$. Let $(U, V)$ be an admissible pair of $\{0,1,\dots, n\}$, with $\mu_1(S_U) <0$.
If $\{ i \mid i> r \}\subseteq U$, where $r=\max \{j \mid j \not\in U\}$ and $r \neq 1$, 
then $\mu_n (S_U  T_V) \leqslant (n-2)^2p_r(1)$.  
\end{prop}

\begin{proof}
For $\mu_n(T_V)$, we see that the maximum absolute value is achieved by taking the sum of all eigenvalues of the same sign which corresponds to one of the two sums below
$$
   \sum_{\stackrel{k \leqslant r}{k \ \textrm{is even} }} (-1)^{k} \binom{n}{k} \leqslant \frac{r}{2}\binom{n}{r};
\quad
  \sum_{\stackrel{k \leqslant r}{k \ \textrm{is odd}}} (-1)^{k} \binom{n}{k} \leqslant \frac{r}{2}\binom{n}{r}.
$$
Both inequalities follow from the fact that Corollary~\ref{cor:set} implies $r \leqslant \frac{n}{2}.$

If $\{ i \mid i>r\} \subseteq U$, then an upper bound on the eigenvalue of $\mu_n( S_U )$ is achieved by first summing all $p_i(n)$ where $i>r$. Then, this sum is added to the sum of all $p_i(n)$, with $i \leqslant r$, of the same sign. Using Equation~(\ref{eq:partial}), we have
\begin{align*}
  | \mu_n( S_U ) | &\leqslant  |\sum_{i>r}p_i(n)| + |\sum_{\stackrel{i \leqslant r}{i \ \textrm{is even} }}p_i(n)|\\
 &\leqslant | -\sum_{i\leqslant r} p_i(n) | + \frac{r}{2} \binom{n}{r}
 = \left| \sum_{i\leqslant r} (-1)^i\binom{n}{i} \right| + \frac{r}{2} \binom{n}{r}
 =  \left| (-1)^r\binom{n-1}{r} \right| + \frac{r}{2} \binom{n}{r}\\
 &\leqslant n\binom{n-1}{r}.
\end{align*}

Consequently, under the given hypothesis, $|\mu_n(S_U T_V) | = |\mu_n(S_U) \ \mu_n(T_V) | \leqslant \frac{nr}{2}\binom{n}{r}\binom{n-1}{r}$.  

We get a lower bound on $p_r(1)$ by simplifying the final factor in Equation~(\ref{eq:Hammingeigenv1}) for $i=r$ using the fact that $r \leqslant \frac{n}{2}$ and $q>n$:
$$
\frac{(n-r)}{r}(q-1)-1  \geqslant   \frac{(n-r)}{r}(n-1)-1  \geqslant  2(n-r)-2  \geqslant   2(n-r-1).
$$

To prove the statement, we show the ratio of $(n-2)^2p_r(1)$ and $| \mu_n(S_U T_V)|$ is at least $1$:
\begin{align*}
\frac{(n-2)^2p_r(1)}{| \mu_n(S_U T_V) | }
&\geqslant
\frac{(n-2)^2(q-1)^{r-1}\binom{n-1}{r-1} \left( (\frac{n-r}{r})(q-1)-1 \right)}{ \frac{nr}{2} \binom{n}{r} \binom{n-1}{r}} \\
&\geqslant 
\frac{2 \ (n-2)^2(q-1)^{r-1}\binom{n-1}{r-1} (n-r-1)  }{ \frac{nr}{2}(\frac{n}{r})\binom{n-1}{r-1}\binom{n-1}{r}} \\
& \geqslant 
\frac{4  (n-2)^2(q-1)^{r-1}  (n-r-1)  \ r! }{n^2(n-1) \cdots (n-r)} >1.
 \end{align*}

In all cases, we have $| \mu_n(S_U) \mu_n(T_V) | < (n-2)^2| p_r(1) |$.
\end{proof}

\begin{thm} \label{thm:Hamminnon}
If $q>n$ and $n\geqslant 5$, then $H( n, q)$ does not admit a $\lambda$-fold near-factorization. 
\end{thm}

\begin{proof}
With the aid of a computer, one can verify that $\mu_1(S_U T_V) \neq \mu_n(S_U T_V)$ for all admissible pairs over all $(q,n)$ such that  $q<17$ and $q > n \geqslant 5$. 

We will assume $\mu_n(S_U)<0$. If $q>17$, then Corollary~\ref{cor:set} implies that, if $(S_U,T_V)$ is a $\lambda$-fold near-factorization of $H(n, q)$, then $\{\frac{n}{2}, \frac{n}{2}+1, \ldots, n\} \subseteq U$. Since $U$ and $V$ are disjoint, and $|V|\geqslant 2$, then $j \in V$ for some $j \in \{1,2,\ldots, \frac{n}{2}\}$. Therefore, by Lemma~\ref{lem:n2}, 
$|\mu_1(T_V)|\geqslant (n-2)^2$.  

Let $r=\max\{i \mid i \not \in U\}$ and $r \neq 1$. By Proposition~\ref{prop:uppern}, if $\{ i \mid i> r\}\subseteq U$ and $1<r \leqslant \frac{n}{2}$, then $|\mu_n(S_U)\mu_n(T_V)|< (n-2)^2 p_r(1)$. By Corollary~\ref{cor:lower1}, $p_r(1)<|\mu_1(S_U)|$ and thus, since $|\mu_1(T_V)| \geqslant (n-2)^2$, $|\mu_1(S_U T_V)|>(n-2)^2|\mu_n(S_U)|\geqslant |\mu_n(S_UT_V)|$. Therefore, $\mu_1(S_U T_V)\neq \mu_n(S_U T_V)$, a contradiction. Thus, $r \in U$ for each $1 < r \leqslant \frac{n}{2}$. As a result,  $S_U=J-I$ and $T_V=I$ which means that $(U,V)$ is not admissible.

Finally, consider the case $\max\{i \mid i \not \in U\}=1$. Here $| \mu_n(S_U) \mu_n(T_V) |\leqslant \frac{n(n-1)}{2}$ and $p_1(1)=(n-1)(q-1)-1$ and thus, $(U, V)$ is also not a near-factorization. 
\end{proof}


\section{Near-factorizations in folded and halved cubes}\label{sec:FFHamming}

In this section, we consider the existence of near-factorizations in two association schemes related to $H(n,2)$, namely the \emph{folded cube} and the \emph{halved cube}. Both schemes can be defined relative to the $H(n,2)$. 

\subsection{The folded cube}
\label{subsec:FoldedCube}

In Subsection~\ref{subsec:Hamming2}, it was shown that, when $n$ is even, the scheme $H(n,2)$ admits a near-factorization. However, when $n$ is odd, we conjecture that $H(n,2)$ does not admit a near-factorization. This conjecture has been verified for all odd $n$ such that $n \leqslant 15$. When $n$ is odd, we can, however, establish the existence of a near-factorization in the $n$-folded cube association scheme which is is an association scheme derived from $H(n,2)$. 

We rely on the definition of the folded Hamming scheme given in~\cite[Section 9.2 D]{BCN}.  First, we define the set of points of the folded Hamming scheme. The \emph{complement} of a point $v$ of $H(n,2)$ to be the vertex formed by switching the ones and zeros in $v$. The pair $v$ and $v^c$ form a pair of \emph{antipodal points}. 

\begin{defn}
Let $V$ be the set of words of length $n$ over an alphabet of order $2$ and let $H(n,2)$ be the Hamming scheme over $V$. The \emph{$n$-folded cube} association scheme, $H'(n,2)=\{A_0, A_1, A_2, \ldots, $ $ A_{\lfloor\frac{n}{2}\rfloor}\}$, is the association scheme on set of points $\{\{ v, v^c \} \mid v \in V \} $. Two points $\{v, v^c\}$ and $\{w, w^c\}$ are $i$-related, with $i \in \{0, 1,2,\ldots,\lfloor \frac{n}{2} \rfloor \}$, if and only if $v$ is $i$-related to $w$ or $w^c$ in $H(n,2)$. 
\end{defn}

The spectrum of the folded cube association scheme can be computed from the spectrum of the Hamming scheme by simply summing the eigenvalues of $i^{th}$ and $(n-i)^{th}$ matrices on even eigenspaces. This statement implies the following lemma. 

\begin{lem}
Let $j \in \{ 0, \ldots, \frac{n-1}{2}\}$ and $n$ be odd. The eigenvalue of $A_i$ on the $j^{th}$-eigenspace of the folded cube association scheme is
$$
\mu_{j}(i) = \frac12	
\left( p_i(2j) + p_{n-i}(2j) \right) = 
	\frac12 \sum_{r=0}^{2j} (-1)^r 
	\left( \binom{n-2j}{i-r} + \binom{n-2j}{n-i-r} \right) 
	 \binom{2j}{r}.
$$

\end{lem}

Below, we define the pairs of subsets of relations that give rise to a near-factorization in the $n$-folded cube association scheme. If $n \equiv 1 \pmod{4}$:
\[
U_1 = \{ i \ | \ i \equiv 0,1 \ (\textrm{mod}\ 4)\}; 
\quad
V_1=\{i \ | \ i \equiv 2,3\ (\textrm{mod}\ 4)\}.
\]
\noindent If $n \equiv 3 \pmod{4}$:
$$
U_2 = \{ i \ | \ i \equiv 0,3 \ (\textrm{mod}\ 4) \};
\quad 
V_2=\{i \ | \ i \equiv 1,2 \ (\textrm{mod}\ 4)\}.
$$
Matrices $S_1= S_{U_1}$, $T_1=T_{V_1}$, $S_2 = S_{U_2}$, and $T_2 = T_{V_2}$ are defined in the same way as in Section~\ref{sec:FFHamming}.

\begin{thm}
The folded cube association scheme admits a nontrivial $2^{n-3}$-fold $(2^{n-1}, 2^{n-1}-1)$ near-factorization realized by $S_1$ and $T_1$ when $n \equiv 1 \pmod{4}$, and by $S_2$ and $T_2$ when $n \equiv 3 \pmod{4}$.
\end{thm}   

\begin{proof}
Let $n \equiv 1 \pmod{4}$, then $i \in U_1$ if and only if $n-i \in U_i$; likewise for $V_i$. Similarly, if $n  \equiv 3 \pmod{4}$, then the same statements hold for $U_2$ and $V_2$. Below, we consider the case $n \equiv 1 \pmod{4}$ (the proof for $n \equiv 3 \pmod{4}$ is identical, so it is omitted).

Recall that $i \leqslant \frac{n}{2}$. Therefore,
\begin{align*}
\mu_j(S_1)=\sum_{i \equiv 0,1 \Mod{4} ) } \mu_{j}(i) 
&=
 \sum_{\stackrel{ i \equiv 0,1 \Mod{4} )}{i \leqslant n/2} }\left( \frac12 \sum_{r=0}^{2j} 
     (-1)^r   \left( \binom{n-2j}{i-r} + \binom{n-2j}{n-i-r} \right)   \binom{2j}{r} \right)   \\
&=
\frac12  \sum_{ i \equiv 0,1\Mod{4} )} \mbox{}  \sum_{r=0}^{2j}  (-1)^r   \binom{n-2j}{i-r}    \binom{2j}{r}   
\end{align*}
and, similarly,
\begin{align*}
\mu_j( T_1)=\sum_{i \equiv 2,3 \Mod{4} ) } \mu_{j}(i) 
&=
\frac12  \sum_{ i \equiv 2,3\Mod{4} ) }  \mbox{}  \sum_{r=0}^{2j}  (-1)^r   \binom{n-2j}{i-r}    \binom{2j}{r}  .
\end{align*}
By Lemma~\ref{lem:sum0}, when $j=0$,
\begin{align*}
\mu_0( S_1 )\mu_0( T_1 )&=\frac14  \left(\sum_{i \equiv 0,1 \Mod{4} } \binom{n}{i} \right)
 \left( \sum_{i \equiv 2,3 \Mod{4} ) }  \binom{n}{i} \right)
\\ 
& = 2^{ n-3 } (2^{n-1} - 1) .   
\end{align*}
Next, we compute the product of the eigenvalues on the $j$-eigenspace for $j>0$, still with the assumption that $n \equiv 1 \pmod{4}$:
 \begin{align}  
 \label{eq:prod}\mu_j(S_1 T_1) &=
   \frac14  
\left(  \sum_{r=0}^{2j}  (-1)^r \binom{2j}{r}  \sum_{ i \equiv 0,1 \Mod{4} }  \binom{n-2j}{i-r} \right)    
\left(  \sum_{r=0}^{2j}  (-1)^r \binom{2j}{r} \sum_{ i \equiv 2,3 \Mod{4} }  \binom{n-2j}{i-r}  \right).
 \end{align}

Our next step is to simplify each term in Equation~(\ref{eq:prod}). To do so, one must consider four cases, one for each $n-2j \equiv 1,3,5,7 \pmod{8}$. Below, we shall assume that $n-2j \equiv 1 \pmod{8}$.  If $n \equiv 1, 5 \pmod{8}$, then $j$ is always even. Therefore
 \begin{align}
\sum_{r=0}^{2j} \left( (-1)^r \binom{2j}{r} \sum_{ i \equiv 0,1 \Mod{4}}  \binom{n-2j}{i-r} \right) 
&=
\sum_{r=0}^{2j} \left( (-1)^r   \binom{2j}{r}  \sum_{ k \equiv 4-r,1-r \ (\textrm{mod}\ 4)}     \binom{n-2j}{k} \right) \\ \label{eq:step1}
&=2^{n-2j-1}  \sum_{r= 0}^{2j}   (-1)^r \binom{2j}{r} + 2^{\frac{n-2j-1}{2}} \sum_{r= 0}^{2j}  (-1)^r \binom{2j}{2r} \\ \label{eq:step2}
&=(-1)^{j/2 }  2^{(n-1)/2}.\nonumber
\end{align}
(To go from (eq:step1) to (eq:step2) first expand the second summation using Lemma~\ref{lem:sum0} and evaluate the trigonometric functions.)
Similarly, we expand the second term ({eq:prod}) to get
 \begin{align*} 
  \sum_{r=0}^{2j}  (-1)^r \binom{2j}{r}  \sum_{ i \equiv 2,3 \Mod{4} }  \binom{n-2j}{i-r}
  & = - 2^{(n-2j-1)/2} \sum_{r= 0}^{j}  (-1)^r \binom{2j}{2r} = (-1)^{j/2+1}  2^{(n-1)/2}.
 \end{align*}
Now we can easily compute the product of the two terms in Equation~(\ref{eq:prod}), which is $\mu_j(X_{S_1}X_{T_1})=-2^{n-3}$.

A similar reasoning applies to $n-2j \equiv 3,5,7 \pmod{8}$.
\end{proof}


\subsection{The halved cube}
\label{subsec:HalvedCube}

In the Hamming scheme $H(n,2)$ the matrix $A_i$ is the adjacency matrix of the graph $H(n,2,i)$ in which two words are adjacent if they differ in exactly $i$ positions. If $i$ is even, then this graph is not connected---one connected component contains all the sequences with even weight, and the other contains all with odd weight. The component of $H(n,2,2)$ containing all sequences with even weight is called the \emph{halved cube graph}, it is denoted $\frac12Q_n$. The halved cube is a distance-regular graph (this follows from the fact that it is the bipartite half of a distance-regular graph). Consequently, the distance graphs of $\frac12Q_n$ form an association scheme, $\{A_0,A_1, \dots, A_{\lfloor \frac{n}{2} \rfloor} \}$---we call this the \emph{halved cube association scheme}, and each $A_i$ is the adjacency matrix of the component of $H(n, 2, 2i)$ in which all the sequences have even weight. For more about halved cube graphs, see \cite[Section 9.2 D]{BCN}.

\begin{cor}\label{cor:Kraw2}
The eigenvalues of the $A_i$ in the halved cube association scheme are
\[ 
	p_i(j)=\sum_{r=0}^j (-1)^r \binom{n-j}{2i-r}\binom{j}{r},\quad j \in \{0, \dots, \lfloor \frac{n+1}{2} \rfloor \}.
\]
\end{cor}

For $n$ odd, the halved cube scheme has a near-factorization that is similar to the near-factorization for $H(n,2)$ with $n$ even. In this case, we let $U= \{i \mid i\equiv  0 \pmod{4} \}$ and $V= \{i \mid i\equiv  2 \pmod{4} \}$ and $S_U$ and $T_V$ be the two corresponding  matrices. In Theorem~\ref{thm:hfcb}, we show that $(S, T)$ is a near-factorization of the half cube association scheme. To prove this, we need a formula for the eigenvalues similar to the ones in Section~\ref{subsec:Hamming2}. See Appendix \ref{app:calc} for computations.

\begin{restatable}{lem}{foldedEvalues}
\label{lem:foldedEvalues}
Let $n$ be odd. For $j \in \{1, \ldots, \frac{n+1}{2}\}$, $\mu_j(S_U)= \pm 2^{ (n-3)/2 }$, and $\mu_j(T_V) =\mp 2^{ (n-3)/2 }.$
\end{restatable}

\begin{thm}
\label{thm:hfcb}
If $n$ is odd, then the halved cube association scheme admits a $2^{n-3}$-fold 
$( 2^{n-2} -   2^{\frac{n-3}{2} } , \, 2^{n-2}+  2^{\frac{n-3}{2} } )$
near-factorization realized by $(S_U,T_V)$.
\end{thm}

\begin{proof}
First, observe that
\begin{align*}
\mu_0(S_U) \mu_0(T_V)&=\left(\sum_{i  \textrm{ even} } \binom{n}{2i}\right)\left(
\sum_{i  \textrm{ odd} }  \binom{n}{2i}\right)\\
&=
\sum_{i=0}^{\lfloor \frac{n}{2}\rfloor} \binom{n}{4i}
\sum_{i=0}^{\lfloor \frac{n}{2}\rfloor} \binom{n}{4i+2} \\
&=
\left( 2^{n-2}+ (2^{\frac{n}{2}-1})\cos\bigg(\frac{n\pi}{4}\bigg) \right)
\left( 2^{n-2}- (2^{\frac{n}{2}-1})\cos\bigg(\frac{n\pi}{4}\bigg) \right) \\
&= 2^{ n-3 } (2^{n-1} - 1).
\end{align*}
The penultimate equality follows from Lemma~\ref{lem:sum0}, and the last one follows from $\cos(\frac{n\pi}{4}) = \pm \frac{1}{\sqrt{2}}$ for $n$ odd.
For $j \neq 0$, from Lemma~\ref{lem:foldedEvalues} we have 
$\mu_j(S_U T_V) =  \pm 2^{ (n-3)/2 } \mp 2^{ (n-3)/2 } = - 2^{n-3}$.
Therefore, $(S, T)$ is the desired near-factorization. 
\end{proof}


\section{Non-existence of $\lambda$-fold near-factorizations in Johnson schemes}
\label{sec:Johnson}

In this section, we establish the non-existence of a $\lambda$-fold near-factorization for certain cases of the Johnson scheme by employing the same approach as that for the Hamming scheme in Subsection~\ref{SS:qnHam}.  

Let $n \geqslant k$. The \emph{Johnson scheme}, denoted $J(n, k)$, is the association scheme whose point sets are the $k$-subsets of an $n$-element set. Two subsets are $i$-related if the size of the intersection of the two corresponding subsets is $k-i$. Thus, $J(n,k)$ has $k$ nontrivial classes and $A_0$ is the identity and $A_k$ is the adjacency matrix of the Kneser graph.  We will make the standard assumption that $2k \leqslant n$. This association scheme is both $P$-polynomial and $Q$-polynomial, so there is an ordering of both the matrices in the association scheme and the idempotents. This gives rise to an ordering of the eigenspaces called the \emph{standard ordering} that we will use in this section.

The spectrum of $J(n, k)$ can be derived by using the \emph{Eberlein polynomials}. Namely, the eigenvalue of $A_i$ realized by the $j^{th}$-eigenspace is

\begin{equation}
\label{eq:eber}
  p_i(j)=  \sum_{r=0}^{i} (-1)^{i-r}  \binom{k-r}{i-r}  \binom{k-j}{r}  \binom{n-k-j+r}{r}.
\end{equation}
     
In this section, we will employ the same notation as that used in Section~\ref{sec:Hamming} to refer to the eigenvalues and eigenspaces of matrices $S_U$ and $T_V$. An \emph{admissible pair of $J(n,k)$} of subsets of relations is defined analogously. Similar to the proof of Theorem~\ref{thm:Hamminnon}, we consider all admissible pairs $(U,V)$ and show that 
$\mu_1(S_U T_V) \neq\ \mu_k(S_U T_V)$. Corollary~\ref{cor:2ndeval} and Theorem~\ref{thm:char2} will jointly imply non-existence of a nontrivial near-factorization. 

We will use an expression for the eigenvalues realized by the $1^{st}$-eigenspace from Section~13.9 of~\cite{MR3497070}. Using a quotient matrix argument, Godsil and Meagher~\cite{MR3497070} derived the following expression for $p_i(1)$:
\begin{equation}
\label{eq:eigen1}
    p_i(1)
      = \binom{n-k-1}{i} \binom{k}{k-i} - \binom{k-1}{k-i}\binom{n-k}{i}
      = \binom{k}{k-i}\binom{n-k}{i}\left(\frac{n-k-i}{n-k}-\frac{i}{k}\right).
\end{equation}

As for the eigenvalues realized by the $k^{th}$-eigenspace, there is precisely one non-zero term in Equation~(\ref{eq:eber}) when $j=k$. Therefore 
\begin{equation}
\label{eq:eigen2}
    p_i(k)=(-1)^i \binom{k}{i}.
\end{equation}

When $n$ is sufficiently large relative to $k$, it can be shown that the $1^{st}$-eigenspace exhibits the same sign pattern as in the Hamming scheme. 

\begin{lem}
\label{lem:pos}
If $n \geqslant k^2$, then $p_i(1)<0$ if and only if $i=k$.     
\end{lem}

\begin{proof}
From Equation~(\ref{eq:eigen1}), if $i = k$
 $$
 p_k(1) =  \binom{n-k-1}{k}   - \binom{n-k}{k} < 0.
$$ 
Since the assumption $n \geqslant k^2$ implies $\frac{n-k-i}{n-k} \geqslant \frac{i}{k}$ when $i<k$, the lemma follows. 
\end{proof}

By Corollary~\ref{cor:2ndeval}, if $(U, V)$ is an admissible pair that gives rise to a $\lambda$-fold near-factorization of $J(n, k)$, then $\mu_1(S_U T_V)$ must be negative. This means that exactly one of $S_U$ or $T_V$ has a negative eigenvalue on the $1^{st}$-eigenspace. Without loss of generality, assume that $k \in U$, and thus, we have $\mu_1(S_U)<0$. Using the same logic as in Corollary~\ref{cor:lower1}, this implies if $i \not\in U$, then $| \mu_1(S_U) | \geqslant p_i(1)$.

\begin{thm}
\label{thm:nonJohn}
If $n > k^2+k$ then the Johnson scheme $J(n, k)$ does not admit a nontrivial $\lambda$-fold near-factorization.
\end{thm}

\begin{proof}
Let $(U,V)$ be a nontrivial admissible pair such that $k \in U$ and define $r=\max\{i \mid i \not\in U \}$ (note that $1\leqslant r \leqslant k-1$). Then 
\begin{equation} \label{eq:lowerlambda1}
 | \mu_1(S_U T_V) | 
 > 
 | p_r(1) | 
=
\binom{k}{r}\binom{n-k}{r} 
    \left(\frac{n-k-r}{n-k}-\frac{r}{k}\right) 
=
\binom{k}{r} \binom{n-k}{r} \frac{n(k-r)-k^2}{k(n-k)}.
\end{equation}
By a similar reasoning as that given in the proof of Proposition \ref{prop:uppern}, it can also be shown that $\mu_k(S_U T_V)\leqslant \frac{kr}{2}\binom{k}{r}\binom{k-1}{r}$. We will show that $\mu_k(S_U T_V) < \mu_1(S_U T_V)$ by proving this upper bound on
$\mu_k(S_U T_V)$ is smaller than the lower bound in Equation~(\ref{eq:lowerlambda1}). First, direct calculation shows this holds for $r=1,2,3$. For larger $r$ we use the approximation 
$\frac{n(k-r)-k^2}{k(n-k)} \geqslant \frac{1}{n-k}$, then 
\begin{align}
\label{eq:ratio}
\frac{p_r(1)}{\mu_k(S_U T_V)}
 \geqslant  
\frac{\binom{k}{r} \binom{n-k}{r} \frac{1}{n-k} }{\frac{kr}{2}\binom{k}{r}\binom{k-1}{r}} 
 =   \frac{ 2 (n-k-1)(n-k-2) \cdots (n-k-r+1) }{ kr(k-1)(k-2) \cdots (k-r) }.
\end{align}
A lower bound for the numerator is given by 
$$
(n-k-r+1)^{r-1} \geqslant (k^2-k+2)^{r-1} > (k-1)^{2(r-1)},
$$
and the denominator is at most $k(k-1)^{r+1}$. Thus Equation~(\ref{eq:ratio}) is at least $2(k-1)^{r-3}/k$. If $r \geqslant 4$, then this implies Equation~(\ref{eq:ratio}) is at least 1, therefore $|\mu_1(S_U T_V)|>|\mu_k(S_U T_V)|$, and the pair $(U,V)$ is not a near-factorization. Hence, $|\mu_1(S_U T_V)|>|\mu_k(S_U T_V)|$ for all admissible pairs $(U,V)$ and so, by Corollary~\ref{cor:2ndeval}, $(U, V)$ does not give rise to a near-factorization of $J(n, k)$. By Theorem~\ref{thm:char2}, under the given hypothesis $J(n, k)$ has no nontrivial near-factorization. 
\end{proof}

\section{Non-existence of $\lambda$-fold near-factorizations in Grassmann schemes}\label{sec:grassman}

We will now consider a scheme closely related to the Johnson scheme known as the Grassmann scheme.  \emph{The Grassmann scheme}, denoted $J_q(n, d)$, is the association scheme whose set of points is the set of  $d$-dimensional subspaces of a vector space of dimension $n$ over finite field $\mathbb{F}_q$, where $q$ is necessarily a prime power. Two subspaces are $i$ related if the subspace that arises from their intersection is of dimension $d-i$. The scheme $J_q(n,d)$ has $d$ classes and is both $P$-polynomial and $Q$-polynomial. Namely, $J_q(n,d)$ is metric and cometric relative to the Grassmann graph, which is the graph in which two subspaces are adjacent if their intersection has dimension $d-1$, and this imposes natural orderings of relations and eigenspaces of the Grassmann scheme.  

Just as we saw in the Johnson scheme, the scheme $J_q(n, n-d)$ and $J_q(n, d)$ are isomorphic. Therefore, we will assume that $n \geqslant 2d$. 

The number of $d$-dimensional subspaces of an $n$-dimensional vector spaces over $\mathbb{F}_q$ is given by the \emph{$q$-binomial coefficient}, or \emph{Gaussian binomial coefficient}, denoted by $\qbinom{n}{d}$ and computed as follows:
$$
\qbinom{n}{d}=\prod_{i=0}^{d-1}\frac{q^{n-i}-1}{q^{d-i}-1}.
$$

It can be shown by induction on $n$ that $\qbinom{n}{d}$ is a polynomial in $q$.  Several identities that apply to binomial coefficients have a $q$-analogue. We will introduce these required identities as needed. 

With respect to the natural orderings of relations and eigenspaces, the eigenvalue
realized $X_i$ on the $j^{th}$-eigenspace is $p_{j,i} = \chi_j(A_i)$ of $J_q(n,d)$ are (see \cite{CioabaGupta21})

\begin{equation}
\label{eq:specGras1}
    p_i(j)=\sum_{h=0}^j (-1)^{j-h}q^{i(i-j+h)+\binom{j-h}{2}} \qbinom{j}{h} \qbinom{d-h}{i} \qbinom{n-d-j+h}{n-d-i}.
\end{equation}
We will also need a second equivalent formula, also in~\cite{CioabaGupta21}, 
\begin{equation}
\label{eq:specGras2}
      p_i(j)=  \sum_{h=0}^{i} (-1)^{i-h}  q^{hj+\binom{i-h}{2}}    \qbinom{d-h}{i-h}  \qbinom{d-j}{h}  \qbinom{n-d-j+h}{h}.
\end{equation}

As in the Hamming and Johnson schemes, we will derive a contradiction using the eigenvalues realized by the $1^{st}$ and the $d^{th}$-eigenspaces. 
\begin{lem}
\label{lem:grassmann}
\begin{enumerate}[(i)]
\item The only negative eigenvalue in the $1^{st}$ eigenspace is $P_d(1)= - q^{d^2-d} \qbinom{n-d-1}{d-1}$.\label{grassmann1}
\item If $n > 2d$, $min\{ |P_i(1) | : i=1,\dots,d \} = p_1(1)$.  
\label{grassmann2}
\item The eigenvalues $p_i(d)$ in the $d^{th}$ eigenspace alternate in sign.
\label{grassmann3}
\end{enumerate}
\end{lem}

\begin{proof} 
(\ref{grassmann1}) is \cite[Lemma 5.4(i)]{MR3862932}. 

Since we are assuming $n<2d$, (\ref{grassmann2}) follows from \cite[Lemma 5.4(iii)]{MR3862932}. 

For (\ref{grassmann3}), all but one term of Equation~(\ref{eq:specGras2}) involves a $q$-binomial coefficient equal to zero. Therefore, we can simply the sum to  

$$ p_i(d) = (-1)^i q^{\binom{i}{2}} \qbinom{d}{i}. $$

This alternates in sign, and thus, statement (\ref{grassmann3}) follows.
\end{proof}

\begin{thm} [$q$-Binomial Theorem]
\label{thm:qbin}
For any $n \in \mathbb{N}$:
$$  
   \prod_{i=0}^{n-1} (1+q^iz)=\sum_{i=0}^nq^{\binom{i}{2}}\qbinom{n}{i}z^i. 
$$   
\end{thm}

Using the $q$-Binomial Theorem, we obtain the following binomial identity.

\begin{lem}
\label{lem:sum0q}
Let $d \in \mathbb{N}$. Then 
$$
\sum_{i \ \textrm{even}} q^{\binom{i}{2}}\qbinom{d}{i} = \sum_{i \ \textrm{odd}}q^{\binom{i}{2}}\qbinom{d}{i}=\frac{1}{2}\prod_{i=1}^d(1+q^i).
$$
\end{lem}

The reasoning in the proof of Theorem~\ref{thm:nonGrass} is identical to that of the proof of Theorem~\ref{thm:nonJohn}. This time, we will use the following lower bound on $\qbinom{d}{r}$ that can be found in~\cite{lowq}:

\begin{equation}
\label{eq:lo}
\qbinom{d}{r}\geqslant  q^{d(d-r)}.
\end{equation}

\begin{thm}
\label{thm:nonGrass}
When $n > 2d^2-2d+3$, the Grassmann scheme $J_q(n,d)$ does not admit a nontrivial $\lambda$-fold near-factorization. 
\end{thm}
\begin{proof}
Let $U,V$ be an admissible pair so that $S_U, T_V$ forms a near-factorization. We assume that $\chi_1(S_U) <0$, so Lemma~\ref{lem:grassmann}~(\ref{grassmann1}) implies $d \in U$. Similar to the Johnson and Hamming scheme, we show that $\chi_d(S_U T_V) < \chi_1(S_U T_V)$.

As $d \in U$, $\chi_1(T_V)$ is a sum of positive eigenvalues. Thus $\chi_1(T_V)> P_{1,1}$, and so
\begin{align*} 
| \mu_1(S_U)\mu_1(T_V) | &>  | p_1(1) | \\
         &= \left|  -q \qbinom{d-1}{1} \qbinom{n-d}{1} + \qbinom{d}{1}  \right| \\
         &\geqslant q (q^{(d-2)})(q^{(n-d-1)}) -  \qbinom{d}{1} \\
         &\geqslant q^{(n-2)} - (q^{(d-1)}+ \dots + q + 1) \\
         &\geqslant q^{(n-3)}.
\end{align*}

On the other hand, 
$$
|\mu_d(S_U) \mu_d(T_V)| 
  \le \left( \sum_{i \; even} p_i(d) \right)^2 
   = \frac{1}{4} \prod_{i=1}^d (1+q^i)^2 
   \le \frac{1}{4} \prod_{i=1}^d (2q^i)^2 
   = q^{2d(d-1)} 
   \le q^{2d^2-2d}.
$$  
If $n > 2d^2-2d+3$, then $-\lambda \le -q^{(n-3)} < -q^{2d^2-2d} \le \mu_d(S_U) \mu_d(T_V)$, a contradiction.  This proves the theorem. 
\end{proof}

We conjecture that there does not exist a Grassmann scheme that admits a $\lambda$-fold near-factorization. We have verified this statement for all $(n,d)$ with $d \leqslant 9$ and $\qbinom{n}{d} \leqslant 289$.


\section{Cyclotomic schemes and orbit fusions of thin schemes}\label{sec:cyclo}

Any group can be realized as a thin association scheme; each class in this association scheme corresponds to a single element of the group when realized by the permutation matrices of its Cayley representation. This is a thin scheme, since each class has valency one, and, unless every non-identity element of the group has order 2, this is an asymmetric association scheme. Near-factorizations in this thin association scheme are exactly equivalent to near-factorizations in the group. In this section, we will consider association schemes that are fusions of thin association schemes.
Throughout this section, $\mathbb{F}_q$ will be a finite field of prime power order $q=p^f$ and $\mathbb{F}_q^{\times}$ will denote the multiplicative subgroup in $\mathbb{F}_q$.

In his thesis, Delsarte noted that it is possible to form association schemes over a finite field by using cyclotomy; these are called ``cyclotomic schemes" \cite{BCN}. Cyclotomic schemes are a rich source of association schemes with near-factorizations. As we will see, the near-factorizations in cyclotomic schemes can be discovered by establishing identities satisfied by Gaussian periods.  

\begin{defn}
For $d$ a divisor of $|\mathbb{F}_q^{\times}|$, let $H_d$ be the unique multiplicative subgroup of index $d$ in $\mathbb{F}_q^{\times}$ and partition $\mathbb{F}_q$ by the $d$ cosets of $H_d$ in $\mathbb{F}_q^{\times}$. The \emph{cyclotomic association scheme} $\cyc(q, d)$
is the $d$ class association scheme with vertices the elements in $\mathbb{F}_q^{\times}$ and two elements $g, h$ are adjacent in class $i$ if and only if $g-h$ is in the $i^{th}$ coset. 
 \end{defn}
 

Cyclotomic schemes are part of the larger class of \emph{orbit fusion schemes}.  If $G$ is an abelian group of order $n$ in its Cayley representation, and $A$ is a subgroup of $\aut(G)$, then by fusing the $n \times n$ matrices in each $A$-orbit on $G$, we obtain the matrices of an association scheme with point set $G$. We will denote this orbit fusion scheme by $(G,G^A)$, also known as the \emph{orbit Schur ring on the group $G$}. As we will see, many of the Schurian association schemes exhibiting near-factorizations in Section~\ref{sec:schurian} turn out to be orbit fusion schemes on a group $G$ that is either cyclic or elementary abelian. 

\subsection{Cyclotomic schemes of prime order}

Let $p$ be a prime and $\zeta_p$  a primitive $p$-th root of unity in $\mathbb{C}$. The irreducible characters of $C_p = \langle g \rangle$ are the functions $\chi_k(g^i) = \zeta_p^{ki}$ for $k=0,1, \dots, p-1$.  Since $\aut(C_p)$ is cyclic of order $p-1$, its subgroup $H_d$ of index $d$ is generated by $g^a$, where $a$ is congruent modulo $p$ to $\frac{p-1}{d}$. This subgroup $H_d$ acts on the elements of $\mathbb{F}_p$ and the orbits of this action correspond to the classes in the cyclotomic scheme; the zero of $\mathbb{F}_p$ stays fixed, so there is an orbit of size 1. All other orbits correspond to a coset of $H_d$ in $\mathbb{F}_p^{\times}$. This implies the irreducible characters of $\cyc(p, d)$ are the restrictions of the $\chi_k$ to the subalgebra of $\mathbb{C}G$ generated by the sums of the orbits of $H_d$ on $C_p$. It follows that the value of $\chi_k$ on the orbit sum 
$$
	\displaystyle{ \sum_{r=0}^{\frac{p-1}{d} - 1 } g^{ia^r} }
$$
is the Gaussian period 
$$
	\eta_{p, d, ik} = \sum_{r=0}^{\frac{p-1}{d} - 1 }  \zeta_p^{ika^r}.
$$  
The trivial character $\chi_0$ takes the value $|H_d|$ on non-identity elements of $\cyc(q, d)$, and the value of 1 on the identity element. This further implies all of the nontrivial irreducible characters of $\cyc(p, d)$ will be Galois conjugate, the Galois group being $Gal(\mathbb{Q}(\eta_{p, d, 1}) / \mathbb{Q}) \simeq \aut(C_p)/H_d$, with order $d$.

It follows that every near-factorization of a cyclotomic scheme implies and is implied by a certain identity on the Gaussian periods appearing in its character table.  For example, if $\lambda$ is a positive integer for which $d^2=\lambda(p-1)$ and the Gaussian period identity $\eta_{p,d,1} \eta_{p, d, k} = - \lambda$ holds, then it follows that $A_1 A_k = \lambda(J-I)$, because the Gaussian period identity implies the values on all irreducible characters of both sides of the matrix equation will match. Use of such Gaussian period identities in design theory and coding theory has a long history---see \cite{FMX2016} and \cite{MR363657}, for example.

\subsection{Symmetrizations of thin schemes}\label{subsec:SymofThin}

When $G$ is an abelian group, the map $g \mapsto g^{-1}$, $g \in G$, is an automorphism of $G$ which generates a subgroup $S$ of order $2$.  The corresponding orbit fusion scheme $(G,G^S)$ is called the {\it symmetrization} of the thin scheme on $G$. The near-factorizations of the symmetrization $(G,G^S)$ of $G$ correspond to \emph{symmetric} near-factorizations of the group $G$.  For an example of this, consider $G$, the additive group of $\mathbb{F}_q$ where $q$ is odd, and $H$ the subgroup of $\mathbb{F}_q^{\times}$ of order $2$ generated by $-1$. Then the symmetrization $(\mathbb{F}_q, \mathbb{F}_q^S)$ agrees with the cyclotomic scheme $\cyc(q,\frac{q-1}{2})$. In \cite{kreher2025lambda}, it was noted that every $1$-fold near-factorization of an abelian group is (translation) equivalent to a symmetric one, (and this is not always the case when $\lambda > 1$). They also noted that, when $G=\langle g \rangle$ is cyclic of order $n$, for any factorization $n-1 = s t$, there is an easy $1$-fold $(s, t)$ near-factorization of $G$ given by 

\begin{equation*}
	(g + g^2 + \dots + g^s)(e + g^{s} + g^{2s} + \cdots + g^{(t-1)s}) = G - e. \label{eq:TrivalNFGroups}
\end{equation*}

This near-factorization is not symmetric; to obtain a symmetric near-factorization, simply translate by an appropriate power of $g$.  If $s = 2k+1$, multiply the first factor by $g^{-(k+1)}$ to make it symmetric, and to compensate, multiply the second factor by $g^{(k+1)}$. The second factor becomes symmetric by Corollary~\ref{cor:prop}(\ref{symmetry}). On the other hand, if $s=2k$, multiply the first factor by $g^{k(t-1)}$, and then to compensate, multiply the second factor by $g^{-k(t-1)}$.

In the next section, we will encounter several examples of $1$-fold near-factorizations of symmetrization schemes of cyclic groups of odd order. The next lemma gives several precise near-factorizations in generalizations of these examples. 

\begin{lem}\label{lem:cyclicGroup}
Suppose $G$ is a cyclic group of odd order $n=2d+1$ generated by $\langle g \rangle$.  Let 
$$
	A_i = g^i + g^{-i}, \quad i\in \{1,\dots, d\}, 
$$ 
be the non-identity adjacency matrices of of its symmetrization scheme $(G,G^S)$.
\begin{enumerate}[(i)]
 	\item If $s=2$, then $(G,G^S)$ has the following $1$-fold $(2,d)$ near-factorization depending on $n$:
	\begin{enumerate}
	 \item  $n=8k+1$, $d=4k$, and  $A_1(A_2+A_3+A_6+A_7+\dots+A_{4k-2}+A_{4k-1}) = J-I$.

	\item $n=8k+3$, $d=4k+1$, and $A_1(A_0+A_3+A_4+A_7+A_8+\dots+A_{4k-1}+A_{4k}) = J-I$.

	\item $n=8k+5$, $d=4k+2$, and  $A_1(A_2+A_3+A_6+A_7+\dots+A_{4k-2}+A_{4k-1}+A_{4k+2}) = J-I$.

	\item $n=8k+7$, $d=4k+3$, and $A_1(A_0+A_3+A_4+A_7+A_8+\dots+A_{4k-1}+A_{4k}+A_{4k+3}) = J-I$.
\end{enumerate}

\item If $s=3$ and $n=6k+1$, then $(G,G^S)$ has a $1$-fold $(3,2k)$ near-factorization $(A_0 + A_1) (A_2+A_5+A_8+\dots+A_{3k-1}) = J-I$. 

\item If $s=4$ and $n=16k+1$, then $(G,G^S)$ has a $1$-fold $(4,4k)$ near-factorization $(A_1+A_3)(A_4+A_5+A_{12}+A_{13}+\dots+A_{8k-4}+A_{8k-3}) = J-I$.
\end{enumerate}
\end{lem}


\section{Near-factorizations in small Schurian schemes}\label{sec:schurian}
 
There is a wealth of examples of association schemes known as \emph{Schurian association schemes}. These arise from the action of a group $G$ on the set of cosets of a subgroup $H \leqslant G$, with very specific properties. Clearly the group $G$ is transitive on the set of cosets, but it will not be transitive on the ordered pairs of cosets; the orbits on this second action are called \emph{orbitals} and they can be used to define an association scheme. The points of these scheme are the cosets of $H$ in $G$, and two cosets $xH, yH$ are $i$-related if $(xH,yH)$ belongs to the $i$-th orbital. Equivalently, we obtain an orbital configuration $(G/H, G /\!\!/H)$ whose set of vertices is $G/H$ and whose relations correspond to the $H$-$H$-double cosets of $G$.
This will be an association scheme when the trivial character of $H$ induces to a multiplicity-free character of $G$~\cite[Section~13.4]{MR3497070}.  An association scheme which corresponds to an orbital configuration is called \emph{Schurian}.   

We searched the GAP~\cite{GAP4} library of primitive groups looking for near-factorizations occurring in (non-thin) Schurian schemes. A primitive group $G$ of degree $n$ in the GAP library is presented as a subgroup of $\sym(n)$ whose action on $\{1,\dots, n\}$ is primitive.  We set $H$ to be the stabilizer of a point in $G$. These groups are expressed in GAP as {\tt PrimitiveGroup(n,i)}. However, we will simply denote the group by the GAP library ID $(n, i)$. 

In searching the GAP library we find several small Schurian schemes that admit near-factorizations.  Many of these near-factorizations are realized in schemes with $2$ classes, which we have already encountered in Section~\ref{sec:2class}.  Table~\ref{Table:Schurian} in the Appendix lists the GAP library IDs of the primitive groups of degree up to $289$, of order less than $10^7$, whose associated Schurian scheme has between $3$ and $9$ classes, and which admit a near-factorization that is not realized in a fusion of the scheme.  

Many of the small Schurian schemes with near-factorizations listed in Table~\ref{Table:Schurian} turn out to be orbit fusion schemes on an abelian regular normal subgroup of the corresponding primitive group.  The following lemma helps us identify these schemes from the structure of the primitive group.  

\begin{lem}\label{lem:orbitfusions}
Suppose $P$ is a primitive group of degree $n$ that is isomorphic to $G \rtimes H$, where $G$ is a normal abelian subgroup of order $n$ acting regularly on $\{1,\dots,n\}$ and $H$ is the stabilizer of $1$ in $P$. Then the Schurian scheme $(P/H, P/\!\!/H)$ is isomorphic to the orbit fusion scheme $(G,G^H)$.  
\end{lem}

\begin{proof} 
Since $P = G \rtimes H$, the set of left cosets of $H$ in $P$ is $\{gH: g \in G\}$.  It follows that the set of $H$-$H$-double cosets of $H$ in $P$ is of the form $\{H, Hg_1H, \dots, Hg_dH \}$ for $d$ distinct non-identity elements $g_1, \dots, g_d \in G$.  Suppose $(xH, yH) \in P/H \times P/H$ for some choice of $x,y \in G$.  Then $(xH,yH)$ lies in the same relation in $P/\!\!/H$ as $(H,x^{-1}yH)$, and this relation corresponds to the double coset $Hg_iH$ if and only if $H x^{-1}y H = Hg_iH$.  It follows that there exists $h, k \in H$ such that $x^{-1}y = h^{-1} g_i k$.  But the latter is equal to $(g_i)^h h^{-1}k$, and for this to be an element of $G$, it must be the case that $h^{-1}k \in G \cap H = \{e\}$, so we can conclude that $h=k$.  Therefore, $x^{-1}y$ lies in the $H$-conjugacy class of $g_i$ in $G$, which implies $(x,y)$ lies in the relation determined by the $H$-orbit of $g_i$ in the orbit fusion scheme $(G,G^H)$.  We can conclude that the $n \times n$ adjacency matrices of these relations are exactly the same.   
\end{proof}

In the following sections, we group the examples in Table~\ref{Table:Schurian} into families and try to derive general results about when Schurian schemes of these types have near-factorizations.
  
\subsection{Schurian schemes isomorphic to symmetrizations of thin schemes of prime order}
\label{subsec:isothinprime}

From Lemma~\ref{lem:orbitfusions}, we can see from the structure of the primitive group that all the examples in Table~\ref{Table:Schurian} with odd prime order $p$ are orbit fusion schemes on $\mathbb{F}_p$ whose point stabilizer $H$ acts as a subgroup index $d$ in $\mathbb{F}_p^{\times}$, so these are cyclotomic schemes of type $\cyc(p,d)$. In particular, when the point stabilizer has order $2$, the cyclotomic scheme $\cyc(p,\frac{p-1}{2})$ is equal to the symmetrization of a cyclic group of order $p$. These schemes admit a $1$-fold $(2,\frac{p-1}{2})$ near-factorization which is given in Equation~(\ref{eq:TrivalNFGroups}, and corresponds to a symmetric near-factorization of the cyclic group of odd order $p$.   

\begin{table}[h]
\begin{center}
\begin{tabular}{| l| c|c|c|l|} \hline
\mbox{Group:} & classes & $\lambda$ & (s,t) & \mbox{Near-factorization} \\ \hline
$(7,2)$ &  3 & 1&(2,3) &      $A_2 (I+A_1)  $ \\ 
$(11,2)$ &  5 & 1 & (2,  5) & $A_3 (I+A_1+A_2) $ \\ 
$(13,2)$ &  6  & 1 & (2,6) & $A_1 (A_2+A_3+A_7) $ \\
$(13,2)$ &  6 & 1 & (3,4) & $(I+A_1)(A_2+A_6) $ \\ 
$(17,2)$ &  8 & 1 &(2,8) &  $A_1 (A_2+A_3+A_6+A_7) $ \\
$(17,2)$ & 8 & 1&(4,4) &   $(A_1+A_3)(A_4+A_5) $ \\
$(19,2)$ & 9  & 1 & (2,9) & $A_5 (I+A_1+A_2+A_3+A_4) $ \\ 
$(19,2)$ &  9 & 1 & (3,6) & $(I+A_1) (A_2+A_5+A_6) $ \\ \hline
$(37,4)$ & 9 & 1& (4,9) & $A_1(I+A_6+A_8)$ \\   
\hline
\end{tabular}
\end{center}
\caption{$1$-fold near-factorizations in symmetrization schemes of cyclic groups of odd prime order.}\label{Table:syms}
\vspace{-1.75em}
\end{table}

For the first eight rows of Table~\ref{Table:syms}, the near-factorizations follow the patterns in Lemma~\ref{lem:cyclicGroup}.  The last scheme in Table~\ref{Table:syms} is a $1$-fold $(4,9)$ near-factorization in $\cyc(37,4)$, which is a fusion of the symmetrization of the cyclic group of order $37$.  If $g$ is a generator of $C_{37}$,   $g \mapsto g^6$ defines an automorphism of of order $4$, so $A_1$ corresponds to the element $g+g^6+g^{-1}+g^{-6}$ in the group algebra $\mathbb{C} \langle g \rangle$.  The other factor in this near-factorization turns out to be 
$$e \ + \ (g^4+g^{-13}+g^{-4}+g^{13})\ + (g^{10}+g^{-14}+g^{-10}+g^{14}).$$  Several times in this section we will encounter a near-factorization in a group that is  equivalent to one realized in an orbit fusion scheme of the group.  

\subsection{4-class cyclotomic schemes of prime order}

Using Lemma~\ref{lem:orbitfusions}, we can see that the $4$-class Schurian schemes of prime order $p$ in Table~\ref{Table:Schurian} are cyclotomic schemes of the form $\cyc(p,4)$.  As these are orbit fusions of $\mathbb{F}_p$ corresponding to the index-4 subgroup of $\mathbb{F}_p^{\times}$, they are fissions of the orbit fusion of $\mathbb{F}_p$ on the index-2 subgroup of $\mathbb{F}_p^{\times}$. Further, since $\mathbb{F}_p^{\times}$ is cyclic, $\cyc(p, d)$ will be a fission of $\cyc(p,2)$ whenever $d$ is even. In particular, 
$\cyc(p,4)$ is a fission of $\cyc(p,2)$, and we claim $\cyc(p,2)$ is isomorphic to a Paley scheme of order $p$.  To see this, consider the non-zero elements of $\mathbb{F}_p$ as the $p-1$ distinct powers of a primitive element $\alpha$. The subgroup of $\mathbb{F}_p^{\times}$ with index $2$ is generated by $\alpha^2$, so the orbits of this subgroup on $\mathbb{F}_p$ are $\{0\}$, the set of quadratic residues, and the set of quadratic non-residues.  So for all odd primes $p$, $\cyc(p,2)$ is the Paley scheme of order $p$.  By the results in Subsection~\ref{subsec:Paley}, the Paley scheme have near-factorizations, and these are inherited by their fission schemes. However, the near-factorizations in Table~\ref{Table:Schurian} are not realized in a fusion of $\cyc(p,d)$.  

\begin{table}[h]
\begin{center}
\begin{tabular}{| l| c|c|c|l|} \hline
\mbox{Group:} & classes & $\lambda$ & (s,t) & \mbox{Near-factorization} \\ \hline
$(37,6)^a$ & 4 & 7&(9,28) & $ A_1(J-A_1^{\top})$ \\     
$(101,7)^a$ & 4   & 19&(25,76) & $A_1(J-A_1^{\top})$ \\
$(197,7)^a$ & 4   & 37&(49,148) & $A_3(J-A_3^{\top})$ \\   \hline
$(13,3)^a$ & 4 & 21 &(4,9) & $(I+A_1)(J-I-A_1^{\top}) $ \\ 
$(109,9)^a$ & 4  & 21&(28,81) & $(I+A_1)(J-I-A_1^{\top} )$ \\   \hline
$(17,3)$ & 4 & 1 & (4,4) & $A_1 A_2 $ \\ 
$(257,13)$ & 4 & 16&(64,64) & $A_1A_4$ \\  \hline
\end{tabular}
\end{center}
\caption{Near-factorizations in 4-class cyclotomic schemes.}\label{Table:PaleyFission}
\vspace{-1.75em}
\end{table}

As mentioned in Section~\ref{sec:cyclo}, near-factorizations of cyclotomic schemes can be determined using Gaussian period identities. Originally discovered by Gauss, there are formulas for biquadratic Gaussian periods $\eta_{p,4,1}$; we will use these formulas to establish nontrivial near-factorizations of $\cyc(p,4)$. We will use that fact that the value of the representation $\chi_k$ on the orbit sum $\sum_{r=0}^{-1+(p-1)/d} g^{ia^r}$ is the Gaussian period 
$$
\eta_{p, d, ik} = \sum_{r=0}^{-1+(p-1)/d} \zeta_p^{ika^r}.
$$  
Using this, we can understand, and generalize, all the examples in Table~\ref{Table:PaleyFission} in three propositions.

\begin{prop} 
\label{prop:4ksquared1}
If $p$ is a prime of the form $p = 4 k^2  +1 $ for odd $k$, then $\cyc(p,4)$ has a nontrivial $\big(\frac{3p+1}{16}\big)$-fold $\big(\frac{p-1}{4}, \frac{3p+1}{4}\big)$ near-factorization realized by $S=A_1$ and $T=J-A_1^{\top}$.  
\end{prop}
\begin{proof}
We show that $A_1 ( J-A_1^{\top} ) = \frac{3p+1}{16} (J-I)$. We can do this by showing the spectra of the two matrices are equal.  Since the valency of $A_1$ is $\frac{p-1}{4}$, the valency of the left hand side must be $\big(\frac{p-1}{4}\big) \big(p - \frac{p-1}{4}\big) = \frac{(p-1)(3p+1)}{16},$ 
which matches the valency of the right hand side. 
Let $\chi_1$ be the irreducible character of $\cyc(p,4)$ with $\chi_1(A_1)=\eta$, where $\eta = \eta_{p,4,1}$. Then $\chi_1(A_1^{\top}) = \bar{\eta}$.  If $\chi_j$, $j>1$ is any another non-valency irreducible character, then $\chi_j(A_1)$ will be $\eta_{p,4,j}$, a Galois conjugate of $\eta$. Since $\chi_j(J)=0$ for all $j>0$, we see that the values of all irreducible characters on both sides of our matrix equation will match as long as $\eta$ satisfies the identity $\eta \bar{\eta} = \frac{3p+1}{16}.$ 

A special case of a recent result of Thas and Zagier~\cite[Theorem 4.2(c)]{MR2349900} proves that Gaussian periods $\eta=\eta_{p,4,j}$, where $j \in \mathbb{Z}^+$ satisfy $\eta \bar{\eta} \in \mathbb{Z}^+$ if and only if 
$p=4k^2+1$ with $k$ odd, and in this case $\eta \bar{\eta} = \frac{3p+1}{16}$.  Thus we conclude that $(A_1, J-A_1^{\top})$ gives a near-factorization in $\cyc(p,4)$ with $\lambda = \frac{3p+1}{16}$ whenever $p$ for which $\frac{p-1}{4}$ is an odd perfect square.   
\end{proof}

The set of primes above are precisely those lying in the quadratic progression $16n^2+16n+5$, so its infiniteness would be a special case of either the Hardy-Littlewood or Bunyakowsky conjectures.   

Next we consider the case where $p = 4k^2 + 9$ for odd $k$, the examples from the table that fall in this case are (13,3) and (109,9).

\begin{prop} 
\label{prop:4ksquare9}
If $p$ is a prime of the form $p = 4k^2 + 9$ for odd $k$, then $\cyc(p,4)$ has a nontrivial $\big(\frac{3p + 9}{16}\big)$-fold $\big(\frac{p+3}{4}, \frac{3(p-1)}{4}\big)$ near-factorization realized by $S=I+A_1$ and $T=J-I-A_1^{\top}$.  
\end{prop}
\begin{proof}
A near-factorization of the form $(I+A_1)(J-I-A_1^\top) = \lambda(J-I)$ is equivalent to 
$(\frac{p+3}{4} -\lambda)J+(\lambda-1)I = A_1+A_1^{\top}+A_1A_1^{\top}$.  

When we apply any nontrivial irreducible character of the scheme (or, equivalently, any nontrivial irreducible character of $C_p$) to this identity, it will hold when the (imaginary) Gaussian period $\eta = \eta_{p,4,j}$ satisfies $\eta + \bar{\eta} + \eta \bar{\eta} = (\lambda - 1) \in \mathbb{N}.$
From~\cite[Theorem 4.2(c)]{MR2349900}, the general formula (originally due to Gauss) gives
$\eta = \frac{\sqrt{p}-1}{4} + i \sqrt{\frac{p-B\sqrt{p}}{8}}$, where $p = A^2+B^2$, $A =2k$ with $k$ odd, and $B \equiv 3 \pmod{4}$.

Working with this, we find the necessary and sufficient condition for $\eta + \bar{\eta} + \eta \bar{\eta} \in \mathbb{Q}$ is $B = 3$, and in this case $\eta + \bar{\eta} + \eta \bar{\eta} = \frac{3p-7}{16}$, which implies $\lambda = \frac{3p+9}{16}$.  Since $p=4k^2+9$ with $k$ odd, this $\lambda$ will be a nonnegative integer.   
\end{proof}

The family of primes for which $\eta_{p,4,1}$ has this property is conjectured to be infinite. The primes less than $10000$ of the form $4k^2 + 9$ for odd $k$ are exactly $13$, $109$, $1453$, $3373$, $3853$, $4909$, $6733$, and $9613$. 

The last two $4$-class cyclotomic schemes arise from primitive groups of the form $C_{(4) 2^\ell + 1} \rtimes C_{2^\ell}$, the near-factorization we detected results from a product of two adjacency matrices, so the corresponding Gaussian period identity has the form $\eta_{p,4,1} \eta_{p,4,j} = - \lambda$
for some $j > 1$ and $\lambda \in \mathbb{Z}^+$. 
Again, we can use Gauss's formula to generalize these examples. 

\begin{prop}
\label{prop:16ksquared1}
Suppose $p$ is a prime of the form $p=16k^2+1$ for some positive integer $k$.  Then there exists a unit $j \pmod{p}$ for which $\eta_{p,4,1} \eta_{p,4,j} = - k^2. $
\end{prop}

\begin{proof}
As $p \equiv 1 \pmod{4}$, the unique quadratic subfield of the cyclotomic field $\mathbb{Q}(\zeta_p)$ is $\mathbb{Q}(\sqrt{p})$.  Since $\mathbb{Q}(\eta_{p,4,1})$ is a cyclic degree $4$ extension of $\mathbb{Q}$, $\mathbb{Q}(\sqrt{p})$ is its unique degree $2$ subfield.  Let $j$ be a unit $\pmod{p}$ whose corresponding Galois automorphism fixes $\sqrt{p}$ and moves $\eta_{p,4,1}$ to $\eta_{p,4,j} \ne \eta_{p,4,1}$.  The Gauss's formula for $\eta_{p,4,1}$ in these cases is $\eta_{p,4,1} = \frac{\sqrt{p}-1}{4} + \sqrt{\frac{p-\sqrt{p}}{8}}.$
Since the Galois automorphism $\zeta_p \mapsto \zeta_p^j$ fixes $\sqrt{p}$ and maps $\sqrt{\frac{p-\sqrt{p}}{8}}$ to its negative, we get $\eta_{p,4,1} \eta_{p,4,j} = (\frac{\sqrt{p}-1}{4})^2-\frac{p-\sqrt{p}}{8} = \frac{-p+1}{16} = - k^2.$
\end{proof}

The primes less than $10000$ of the form $16k^2+1$ for some positive integer $k$ are $17$, $257$, $401$, $577$, $1297$, $1601$, $3137$, and $7057$.  

The near-factorization of $\cyc(17,4)$ results from the product of $\eta_{17,4,1} = \zeta + \zeta^4 + \zeta^{-1} + \zeta^{-4}$ and $\eta_{17,4,2} = \zeta^2 + \zeta^8 + \zeta^{-2} + \zeta^{-8}$ being equal to $-1$, (it is the sum of all the nontrivial seventeenth roots of $1$).  This is also a near-factorization in the symmetrization of $C_{17}$, but not one that arises directly from the formula in Subsection~\ref{subsec:SymofThin}.  

The near-factorization of $\cyc(257,4)$ results from the product of $\eta_{257,4,1}$ and $\eta_{257,4,9}$ (and follows since $3$ is a primitive element of the field $\mathbb{F}_{257}$).  Note that in this case, $\eta_{257,4,1}$ is the sum of the $64$ roots of unity in the orbit of $\zeta_{257}$ under the subgroup of automorphisms generated by $\zeta_{257} \mapsto \zeta_{257}^{81}$. 
 
\subsection{Cyclotomic schemes of prime order with more than 4 classes}
\label{subsec:more4}

There are several more Schurian schemes of prime order in Table~\ref{Table:Schurian}.  Again these are cyclotomic schemes of prime order, so the ones with an even number of classes are fissions a Paley scheme.   

As we can see, the parameters of the near-factorizations in Table~\ref{Table:6or8classCycpd} match the parameters of the near-factorization known to exist in the $2$-class Paley scheme fusion of these cyclotomic schemes. But these near-factorizations are not equivalent in the thin scheme of the cyclic group to one realized in its' Payley scheme fusion, and this is why we see them being minimally realized in a higher-rank fusion of the thin scheme. We demonstrate this with the $(31,4)^a$ example above. Let $C_{31} = \langle g \rangle$.  Since $2$ has multiplicative order $5 \pmod{31}$, the standard basis of the $\cyc(31,6)$ scheme consists of $A_1 = g+g^2+g^4+g^8+g^{16}$, $A_2 = g^5+g^9+g^{10}+g^{18}+g^{20}$, and $A_3 = g^7+g^{14}+g^{19}+g^{25}+g^{28}$, plus their three transposes obtained by acting linearly by group inversion. Then $\cyc(p,2)$ is a fusion of $\cyc(31,6)$ with basis $\{A_0,  A_1+A_2+A_3,  A_1^{\top}+A_2^{\top}+A_3^{\top}\}$. The near-factorization in the Paley scheme fusion results from $(A_0+A_1+A_2+A_3)(A_1+A_2+A_3) = 8(J-I)$; this not equivalent by group automorphisms or translation to the near-factorization in Table~\ref{Table:6or8classCycpd}. 
So these near-factorizations are also realized in the cyclic group of order $p$, have the same parameters as the near-factorization realized in the Payley scheme fusion of its' thin scheme, but are not equivalent in the group to that one.  The minimal fusion of the thin scheme realizing these inequivalent near-factorizations lies strictly between the thin scheme and the Payley scheme, this is the one where our procedure detects them.  It is well-known that a cyclic group of odd prime order can have several $\lambda$-fold near-factorizations with the same parameters that are not equivalent under automorphisms, translation, or inversion in the group (see \cite{kreher2025lambda}).

\begin{table}[h]
\begin{center}
\begin{tabular}{| l| c|c|c|l|} \hline
\mbox{Group:} & classes & $\lambda$ & (s,t) & \mbox{Near-factorization} \\ \hline
$(31,4)^a$ &  6 & 8 & $(15,16)$ &   $(I+A_1+A_2+A_2^{\top})(A_1+A_3+A_3^{\top})$ \\    
$(127,9)^a$& 6 & 32 & (63,64) &      $(I+A_1+A_2+A_5) (A_3+A_4+A_5) $ \\    
$(223,5)^a$ & 6 & 56 & (111,112) & $(I+A_1+A_2+A_5) (A_3+A_4+A_5) $ \\      \hline
$(109,8)$   & 6 & 27 & $(54,54)$ & $ (A_1+A_2+A_5)( A_3 +A_4+A_6)$ \\  
$(113,6)$   & 8 & 28 & $(56,56)$ & $(A_1+A_3+A_7+A_8)(A_2+A_4+A_5+A_6)$  \\ 
\hline
\end{tabular}
\end{center}
\caption{Near-factorizations of small cyclotomic schemes of prime order having 6 or 8 classes.}\label{Table:6or8classCycpd}
\vspace{-1.75em}
\end{table}

We also obtain three additional $6$- or $8$-class cyclotomic schemes of prime order in Table~\ref{Table:Schurian} that have nontrivial near-factorizations. We were unable to generalize the second near-factorization of $\cyc(31,4)$ in Table~\ref{Table:additional6or8classcyclotomic}. The two other near-factorizations in Table~\ref{Table:additional6or8classcyclotomic} are small cases of families of strong external difference families in the cyclic group of prime order that have recently been established in~\cite{BaoEtAl2018} and \cite{li-momihara2025}.  The $3$-fold $(18,18)$ near-factorization of $\cyc(109,6)$ is the smallest in the family of $\frac{p-1}{36}$-fold $(\frac{p-1}{6},\frac{p-1}{6})$ near-factorizations of the cyclic group of prime order $p=108n^2+1$ discovered in~\cite{BaoEtAl2018}.  The $7$-fold $(28,28)$ near-factorization of $\cyc(113,8)$ is the $v=2$ case of the $\frac{p-1}{16}$-fold $(\frac{p-1}{4},\frac{p-1}{4})$ near-factorization of the cyclic group of prime order $p=4v^4+12v^2+1$ discovered in~\cite{li-momihara2025}.  

\begin{table}[h]
\begin{center}
\begin{tabular}{| l| c|c|c|l|} \hline
\mbox{Group:} & classes & $\lambda$ & (s,t) & \mbox{Near-factorization} \\ \hline
$(31,4)^a$ & 6 & 4&(6,20) & $(I+A_1)(A_1+A_2+A_2^{\top}+A_3^{\top})$ \\  \hline
$(109,8)$   & 6 & 3&(18,18) & $ A_1 A_3  $ \\    
$(113,6)$   & 8 & 7&(28,28) & $(A_1+A_2)(A_3+A_4)$ \\   
\hline
\end{tabular}
\end{center}
\caption{Additional nontrivial near-factorizations in $\cyc(p,d)$ with $d=6,8$.} \label{Table:additional6or8classcyclotomic}
\vspace{-1.75em}
\end{table}

Our search found one more asymmetric $8$-class cyclotomic scheme whose near-factorization does generalize to a larger family of such association schemes.

\begin{table}[h]
\begin{center}
\begin{tabular}{| l| c|c|c|l|} \hline
\mbox{Group:} & classes & $\lambda$ & (s,t) & \mbox{Near-factorization} \\ \hline
$(73,7)^a$ & 8 & 8& $(9,64)$ & $A_1(J-A_1^{\top}) $\\
\hline
\end{tabular}
\end{center}
\caption{An asymmetric class 8 cyclotomic scheme whose near-factorization does generalize} 
\end{table}

\begin{lem} 
\label{lem:64bsquared9}
If $p$ is a prime of the form $p=64b^2+9=8d^2+1$ where $b, d \in \mathbb{Z}^+$, then $\cyc(p,8)$ has a $\big(\frac{7p+1}{64}\big)$-fold $\big(\frac{p-1}{8},\frac{7p+1}{8}\big)$ near-factorization.  
\end{lem}
\begin{proof}
From~\cite[Theorem 4.2(d)]{MR2349900}, there is a (conjecturally infinite) family of class $8$ cyclotomic schemes on $C_p$ with a near-factorization of the form $A_1(J-A_1^{\top})$. These primes $p$ must satisfy $p = 64b^2+9 = 8 d^2 +1$, for $b, d \in \mathbb{Z}^+$, and the Gaussian period $\eta=\eta_{p,8,k}$ satisfies $\eta \bar{\eta} = \frac{7p+1}{64}$.   
\end{proof}

These primes in the above lemma are determined by the parameters $b$ and $d$, which correspond to the solutions of the Pell equation $d^2-8b^2=1$.  This equation has infinitely many solutions, but whether there are infinitely many primes $p$ determined by these solutions is open. 
 

\subsection{Schurian schemes algebraically isomorphic to Hamming schemes}
\label{subsec:SchurianHamming}

In the next subsections we encounter families of Schurian schemes isomorphic to orbit fusions on an elementary abelian group of prime power order, or the direct product of a cyclic group of some composite order with itself. The first of these come in the form of examples where the association scheme is algebraically isomorphic to the Hamming scheme, in Section~\ref{sec:Hamming} we have shown that these association schemes have near-factorizations, so their isomorphic Schurian association schemes have the same near-factorizations. Each of these primitive groups is of form $G = (C_3)^d \rtimes H$ where $H$ is the stabilizer of a point, and the corresponding Schurian scheme is an orbit fusion scheme on $\mathbb{F}_{3^{d}}$ with $d$ classes that agrees with 
$H(d,3)$.   

\begin{table}[h]
\begin{center}
\begin{tabular}{| l| c|c|c|l|} \hline
\mbox{Group:} & classes & $\lambda$ & (s,t) & \mbox{Near-factorization} \\ \hline
$(27,[3,4,8])$ & 3 & 4&(8,13) & $A_1(I+A_3) $ \\
$(81, [72,73,91] )$ & 4 & 8&(16,40) & $ A_1(A_2+A_4) $ \\
$(243, [13,15,21,22,23] )$ & 5 & 16&(32,121) & $(I+A_2+A_4)A_5$ \\  
\hline
\end{tabular}
\end{center}
\caption{Small schurian schemes algebraically isomorphic to Hamming schemes}
\label{table9}
\vspace{-1.75em}
\end{table}

The group indexed by $(243, 23)$ in Table~\ref{table9} is the full automorphism group of $H(5,3)$. The other groups in this table with degree $243$ are proper subgroups of the automorphism group, these subgroup produce the same orbit fusion since their point stabilizers all have the same orbits on $\mathbb{F}_{3^5}$.


\subsection{Products from half-set near-factorizations}
\label{subsec:schurianProduct}

Table~\ref{Table:OAFission} lists primitive groups where the Schurian association schemes are fusions of product graphs with nontrivial half-set near-factorizations. The examples in Table~\ref{Table:OAFission} arise from the product construction in Theorem~\ref{thm:halfsetProds} where $\mathcal{A}$ is a strongly-regular graph or tournament with a nontrivial half-set near-factorization. The half-set near-factorization in $\mathcal{A} \otimes \mathcal{A}$ is realized in a $5$-class fusion of the $8$-class tensor product scheme. 

\begin{table}[h]
\begin{center}
\begin{tabular}{| l| c|c|c|l|} \hline
\mbox{Group:} & classes & $\lambda$ & (s,t) & \mbox{Near-factorization} \\ \hline
$(25,5)$ & 5 & 2 & (4,12) &       $ A_1(A_3+A_4) $ \\          
$(169,75)$ & 5 & 18 & (36,84) & $ A_1(A_4+A_5) $ \\     
$(225,[2,3,4,5])$ & 5 & 32 & (64,112) & $A_1(A_2+A_4) $ \\
$(289,76)$ & 5 &32& (64,144) & $A_5 (A_1+A_4) $ \\    \hline
$(49,12)^a$ & 5 & 8 & $(16,24)$ & $(I+A_1+A_2)(A_1+A_4) $ \\ 
$(121,29)^a$  & 5 & 18 & (36,60) & $(I+A_1+A_5)(A_4+A_5) $ \\ 
\hline
\end{tabular}
\end{center}
\caption{5 class association scheme with degree a perfect square.}\label{Table:OAFission}
\vspace{-1.75em}
\end{table}

For the groups in this table with degree a square of a prime, the near-factorization is the nontrivial 
half-set near-factorization arising either from the Paley graph $P(q)$ or the Paley tournament $T(q)$ (depending on whether $q\equiv 1$ or $q\equiv 3$ modulo 4). The only degree of this form is $n=225$, in this case the near-factorization arises as the product construction with the strongly-regular graph with parameters $(15,6 ; 1,3)$---the Johnson scheme $J(6,2)$. This is an example of a near-factorization realized in an association scheme, but not in a group. This follows since the automorphism group of this scheme does not have a subgroup of order 225, which means the scheme is not a fusion of a thin scheme of a group of order 225.

There are two primitive groups whose Schurian association scheme is a $9$-class fission of a Hamming scheme that admit a nontrivial half-set near-factorization; these are recorded in Table~\ref{Table:9classFission}. The first of these is a $2$-fold $(4,24)$ near-factorization of $C_7 \times C_7$, which results from applying the half-set product construction to the product of two copies of the symmetric $1$-fold $(2,3)$ near-factorization of $\cyc(7,2)$.  This will produce a $2$-fold near-factorization with $s=4$ in the tensor square scheme $\cyc(7,2) \otimes \cyc(7,2)$. The second $4$-fold near-factorization in this table results from the half-set product construction using half-set $1$-fold near-factorization of $C_5$ and the nontrivial half-set $2$-fold near-factorization of the fusion of $C_5 \times C_5$ appearing in the first row of the previous table.

\begin{table}[h]
\begin{center}
\begin{tabular}{| l| c|c|c|l|} \hline
\mbox{Group:} & classes & $\lambda$ & (s,t) & \mbox{Near-factorization} \\ \hline
$(49,3)$ & 9 & 2 & (4,24) &       $A_1(A_3+A_5+A_6+A_9)=2(J-I)$ \\
$(125,[3,9])$ & 9 & 4 & (8,62) &  $A_1(A_3+A_4+A_6+A_9)=4(J-I)$ \\ 
\hline
\end{tabular}
\end{center}
\caption{9 class fissions of the Hamming scheme that admit a nontrivial half-set near-factorization} \label{Table:9classFission}
\vspace{-1.75em}
\end{table}

Table~\ref{Table:AdditionalFission} has two more examples of near-factorizations of a Schurian scheme corresponding to an orbit fusion of a finite field of prime-square order.   The Schurian scheme in these cases is not cyclotomic.  The primitive groups involved have are of the form $C_{17}^2 \rtimes H$, with non-abelian point stablizer $H$.  The scheme is a $4$-class orbit fusion of $C_{17}^2 \simeq (\mathbb{F}_{289},+)$ with valencies $1$, $48$, $96$, $96$, and $48$.  This scheme has a fusion $\{ \{A_0\}, \{A_1+A_4\}, \{A_2+A_3\} \}$ which matches the $2$-class scheme corresponding to the orthogonal array graph $OA(17,6)$.  We have been unable to generalize this example to a larger family.

\begin{table}[h]
\begin{center}
\begin{tabular}{| l| c|c|c|l|} \hline
\mbox{Group:} & classes & $\lambda$ & (s,t) & \mbox{Near-factorization} \\ \hline
$(289,[70,81])$ & 4 & 8 & (48,48) & $A_1A_4$ \\   
$(289,[70,81])$ & 4 & 32 & (96,96) & $A_2A_3$ \\ 
\hline
\end{tabular}
\end{center}
\caption{Additional orbit fusions of an $\mathbb{F}_p^2$.} \label{Table:AdditionalFission}
\vspace{-1.75em}
\end{table}


\subsection{Schurian association schemes that are not fusions of the thin scheme of a finite group}
\label{subsec:notFusions}

While most of the small Schurian schemes of more than $2$ classes admitting near-factorizations in Table~\ref{Table:Schurian} correspond to near-factorizations in groups, this is not the case for all of them.  The first Schurian scheme in Table~\ref{Table:Schurian}  whose near-factorization does not occur in a group is that of {\tt PrimitiveGroup(28,1)}.  We know this because no group of order $28$ admits a nontrivial $2$-fold $(3,18)$ near-factorization.   The Schurian scheme in this case corresponds to the Coxeter graph of order $28$, which is a distance-regular graph with intersection array $[3,2,2,1;1,1,1,2]$. It is possible to write out the intersection matrices, and it is straightforward to verify that this association scheme admits these two different near-factorizations. These near-factorizations are in Table \ref{tab:4class}.

\begin{table}[h]
\begin{center}
\begin{tabular}{| l| c|c|c|l|} \hline
\mbox{Group:} & classes & $\lambda$ & (s,t) & \mbox{Near-factorization} \\ \hline
$(28,1)$ & 4 & 2& $(18,3)$ & $ (A_1+A_2)A_3$ \\ 
$(28,1)$ & 4 & 4& $(18,6)$ & $ (A_1+A_4)A_2$ \\
\hline
\end{tabular}
\end{center}
\caption{Minimal realizations of near-factorizations in small 4-class Schurian schemes with order not a prime power.}\label{tab:4class}
\vspace{-1.75em}
\end{table}

\subsection{Projective general and special linear groups}
\label{subsec:ProjectiveGroups}

The other Schurian schemes in Table~\ref{Table:Schurian} that arise from primitive groups are isomorphic to projective general linear or special linear groups over finite fields.  We were not able to find generalizations of these three examples.  First, the group indexed by $(21,1)$ is isomorphic to $\pgl(2,7)$. This is the automorphism group of the line graph of the Heawood graph, a distance-regular graph of diameter $3$ whose intersection array is $[4,2,2;1,1,2]$.  The near-factorization $(I+A_1)A_2 = 2(J-I)$ in this association scheme can be seen immediately from the first intersection matrix of this scheme.  Second, $(55,2)$ is isomorphic to $\pgl(2,11)$ and is the automorphism group of the $5$-class symmetric association scheme defined on pairs of points of the projective line $\pg(1,11)$.  This is a symmetric scheme whose relations have valency $1$, $12$, $12$, $12$, $6$, and $12$, the near-factorization arises from the product of an element of valency $12$ with the sum of the element of valency $6$ and another of valency $12$.  Thirdly $(203,1)$ is isomorphic to $\psl(2,29)$, this group produces a $7$-class symmetric association scheme from its action on the cosets of a self-normalizing subgroup $H \simeq A_5$.  The near-factorization is produced from the product of the unique element of valency $12$ with the sum of the identity plus elements of valency $30$, $30$, $20$ and $20$.  The near-factorizations of the Schurian schemes in Table~\ref{Table:PGLorPSL} are not directly related to one realized in a group of the same order, because the corresponding primitive group has no normal subgroup of order $n$ that acts regularly.

\begin{table}[h]
\begin{center}
\begin{tabular}{| l| c|c|c|l|} \hline
\mbox{Group:} & classes & $\lambda$ & (s,t) & \mbox{Near-factorization} \\ \hline
(21,1)   &  3 & 2 & (5,8) & $(I+A_1)A_2 $ \\  
(55,2)   &  5 & 4 & (12,18) & $A_1(A_3+A_4) $ \\ 
(203,1) &  7 & 6 & (12,101) & $A_6(I+A_2+A_4+A_5+A_7) $ \\
\hline
\end{tabular}
\label{Table:PGLorPSL}
\caption{Minimal realizations of near-factorizations in small Schurian schemes with group isomorphic to $\pgl(2,p)$ or $\psl(2,p)$.} 
\end{center}
\vspace{-1.75em}
\end{table}

\section{Further work}\label{sec:further}
 
 This paper is a first consideration of near-factorizations of association schemes, we encountered several problems that we have not been able to resolve. In this chapter we collect some of the problems.
 
\subsection{2-class association schemes}

All the examples in Table~\ref{Table:SRGcase1}, have $\lambda$ a perfect square. However, it is open which perfect squares can serve as the $\lambda$ in a near-factorization.

\begin{quest}
\label{Q1}
For which values $\theta$ is there is a strongly-regular graph with a $\theta^2$-fold near-factorization?
\end{quest}

Similarly, all the examples in Table~\ref{Table:SRGcase3}, have $\lambda = \theta^2+ \theta$.
We know if $4a+5$ is a prime then there is a symmetric $2$-class association scheme admitting a near-factorization of the form $A_1 A_2 = \lambda(J-I)$, since a conference graph of order $4a +5$ exists.  However, in general, for the non-conference graphs satisfying Case~\ref{SRGcase3}, we do not know how to guarantee for a given $\theta$ that there is a strongly-regular graph admitting a solution with $\lambda = \theta^2+ \theta$. This leads to the following question, similar to Question~\ref{Q1}.
\begin{quest}
For which values $\theta$ is there a strongly-regular graph with a $(\theta^2 + \theta)$-fold near-factorization?
\end{quest}

\subsection{Conjugacy class association schemes}
\label{subsec:CCscheme}
 
One class of association schemes that we did not consider are known as the \emph{group association scheme} or the \emph{conjugacy class scheme}.

\begin{defn}
Let $G$ be a group with conjugacy classes $C_i$ where $C_0 = \{e\}$. For each conjugacy class $C_i$ define a matrix $A_i$ with rows and columns indexed by $G$ and the $(g, h)$ entry equals 1 if $hg^{-1} \in C_i$ and 0 otherwise. Then
$$
\mathcal{A} =\{ A_i \, | \, \textrm{ $C_i$ is a conjugacy class of $G$} \}
$$ 
is an association scheme called the \emph{conjugacy class scheme}.
\end{defn}

If $G$ is abelian, the group association scheme is equivalent to the group, so a near-factorization in the group association scheme is equivalent to a factorization of the group. We will also note another trivial case.

\begin{lem}
Let $G$ be a group that has a conjugacy class $C_d$ that contains exactly one element of order two. Then the conjugacy class association scheme of $G$ has a near-factorization.
\end{lem}
\begin{proof}
Since the class $A_d$ is thin and symmetric, we have $A_d (J-A_d) = J - I.$
\end{proof}

It seems that nontrivial near-factorizations in conjugacy class schemes are very rare. As noted in Section \ref{S:intro}, Gow and Quinlan~\cite{MR2216464} found central difference sets in certain Suzuki $2$-groups, and some central partial difference sets have been found in certain nonabelian $p$-groups~\cite{MR4903728}.  Although we found a $16$-fold $(28,68)$ near-factorization in the conjugacy class scheme of the group $C_2 \times A_5$ of order $120$, we were unable to generalize this example. This leads to the following question.


\begin{quest}
When does the conjugacy class scheme for a group admit a nontrivial near-factorization?
\end{quest}

\subsection{$P$-polynomial association schemes}  

$P$-polynomial association schemes are the association schemes related to distance-regular graphs. The Hamming, Johnson, and Grassmann schemes are all $P$-polynomial, and we have discussed the existence of near-factorizations in these schemes.  The folded cube and halved cube are both distance-regular graph and we considered near-factorization is the corresponding association schemes. We also found that the schemes corresponding to generalized hexagons of order $(q, q)$ always admit a nontrivial near-factorization. Apart from these, we found two sporadic near-factorizations in $P$-polynomial association schemes of diameter more than $2$: one in the Johnson scheme $J(7,3)$ and the other in in the association scheme corresponding to the Coxeter graph of order 28. Since the adjacency algebras of $P$-polynomial association schemes are generated by the adjacency matrices of distance-regular graphs, the intersection matrices of these schemes can be deduced directly from the intersection array of a distance-regular graph, which makes near-factorizations relatively easy to detect in specific cases. Therefore, we propose the following question for further work.

\begin{quest}
What conditions on a distance-regular graph $\Gamma$ imply the (non)existence of a nontrivial near-factorization in the $P$-polynomial association scheme arising from the distance graphs of $\Gamma$?
\end{quest}

\subsection{Schurian association schemes}  

In Subsection~\ref{subsec:SchurianHamming} we saw several examples of groups with the property that the Schurian association scheme is algebraically isomorphic to the Hamming scheme. In the next question we ask if we can characterize when this occurs.

\begin{quest}
For which subgroups $H$ of $GL(d,p)$ is the Schurian association scheme arising from the primitive group $\mathbb{F}_{p^d} \rtimes H$ algebraically isomorphic to the Hamming scheme $H(d, p)$?
\end{quest}

\bigskip
\noindent\textbf{Acknowledgment} Allen Herman is supported by the Natural Sciences and Engineering Research Council of Canada (NSERC) Discovery Grant RGPIN-2026-06558. Karen Meagher is supported by NSERC Discovery Grant RGPIN-2025-04101. Alice Lacaze-Masmonteil and Hermie Monterde are supported by the Pacific Institute for the Mathematical Sciences (PIMS) Post-Doctoral Fellowship.

\appendix

\section{List of examples of strongly-regular graphs with near-factorizations}
\label{app-SRG}

We can use Brouwer's database of parameters for strongly-regular graphs to give the list of strongly-regular graphs of order up to $100$ and the corresponding $\lambda$ satisfying the conditions of Theorem~\ref{thm:SRGs}(\ref{SRGcase1}). The cases where it is unknown if a strongly-regular graph with the given parameters exists are accented with a ${}^*$.

\begin{table}[h]
\begin{center}
\begin{tabular}{lll}
$[ \, (4, 2 ; 0,2) , 1 \, ]$, &  $[ \, (15, 6 ; 1,3), 4 \,  ]$, &  $[ \,  (16, 5 ; 0,2), 4 \,  ]$,  \\ 
$[ \, (16,9 ; 4,6), 4 \,  ]$ &       $[\,  (35,16 ; 6,8), 9 \,  ]$, & $[ \, (36,14; 4,6), 9 \,  ]$, \\
$[ \, (36,20 ; 10,12), 9 \,  ]$ , & $[\,  (40,12 ; 2,4), 9 \,  ]$, & $[ \,  (45,32 ; 22,24), 9 \,  ]$, \\
$[ \, (56,10; 0,2), 9 \,  ]$, & $[\,  (63,30; 13,15), 16 \,  ]$, & $[ \, (64,27 ; 10,12), 16 \,  ]$, \\
$[\,  (64,35 ; 18,20), 16 \,  ]$, &  $[\,  (96,19; 2,4), 16 \,  ]$, &  $[ \,  (96,75 ; 58, 60), 16 \,  ]$, \\ 
$[ \, (99,48 ; 22, 24), 25 ]\,  $, & $[ \, (100,44 ; 18,20), 25 \,  ]$, &  $[ \,  (100,54 ; 28,30), 25 \,  ]$ . 
\end{tabular}
\end{center}
\caption{Parameters of strongly-regular graphs with near-factorization where $\lambda = k-a-1$.
\label{Table:SRGcase1}} 
\vspace{-1.75em}
\end{table}

\begin{table}[h]
\begin{center}
\begin{tabular}{lll}
$ [ \,  (5,2 \,  ; \,   0,1) , 1 \,  ]$, &  $[ \,  (9,4  \,  ; \,  1,2), 2  \,  ]$, &   $[ \,  (13,6  \,  ; \,  2,3), 4  \,  ]$ , \\
 $[ \, (17,8  \,  ; \,  3,4), 4 \,  ] $, &  $ [ \, (25,12  \,  ; \,  5,6), 6  \,  ]$, & $ [ \, (29,14  \,  ; \,  6,7), 7  \,  ]$, \\
 $ [ \, (37,18  \,  ; \,  8, 9), 9   ] $, & $  [ \, (41,20  \,  ; \,  9,10), 10 \,  ] $, & $ [ \, (45,22  \,  ; \,  10,11), 11  \,  ]$, \\
 $ [ \, (49,24  \,  ; \,  11,12) , 11 \,  ]$, & $ [ \, (53,26  \,  ; \,  12, 13), 13  \,  ]$, & $ [ \, (61,30  \,  ; \,  14,15), 15  \,  ]$, \\
 $ [ \, (65,32  \,  ; \,  15,16), 15 \,  ]$ , & $ [ \, (73,36  \,  ; \,  17,18), 18  \,  ]$, & $ [ \,  (81,40  \,  ; \,  19, 20), 20  \,  ]$, \\
  $[ \,  (85,42  \,  ; \, 20,21)^*, 21 \,  ]$, & $ [ \, (89,44   \,  ; \,  21,22) , 22 \,  ]$ , & $ [ \, (97,48  \,  ; \,  23,24), 24  \,  ]$. \\
\end{tabular}
\end{center}
\caption{Parameters of strongly-regular graphs with near-factorization where $\lambda = a+1$.} 
\vspace{-1.75em}
\end{table}

\begin{table}[h]
\begin{center}
\begin{tabular}{ll}
$ (10, 3 ; 0,1), \,  (10,6 ; 3,4)$ &  $\lambda =2$ \\
$ (26,10 ; 3,4), \,  (26,15 ; 8,9)$ &  $\lambda = 6$ \\
$ (50, 7 ; 0,1), \,   (50,42 ; 35,36)$ &  $\lambda = 6 $ \\
$ (50,21 ; 8,9), \,   (50,28 ; 15,16)$  &  $\lambda =  12$  \\
 $ (82, 36 ; 15,16), \,   (82, 45 ; 24, 25)$ &  $\lambda = 20$ \\
 $ (99,14 ; 1,2)^*, \,   (99,84 ; 71,72)^*$  &  $\lambda = 12$
\end{tabular}
\end{center}
\caption{Complementary pairs of parameters for strongly-regular graphs with near-factorization.} \label{Table:SRGcase3}
\vspace{-1.75em}
\end{table}

\section{Calculations from Section~\ref{sec:Hamming}} 
\label{app:calc}

\subsection{Proof of Lemma~\ref{lem:sums}.} 


\begin{proof}
First, suppose $n=8k$ for some positive integer $k$. If $r \in \{0,2\}$, then Lemma~\ref{lem:sum0} implies
$$
\sum_{i~ \equiv~ r \pmod{4}} \binom{n}{i} 
=2^{n-2}+(-1)^{\frac{r}{2}}(2^{\frac{n}{2}-1})\cos(2k\pi)
=2^{n-2}+(-1)^{\frac{r}{2}}(2^{\frac{n}{2}-1}).
$$
If $n=8k$ and $r \in \{1,3\}$, then Lemma~\ref{lem:sum0} implies
$$
\sum_{i ~ \equiv ~ r \pmod{4}} \binom{n}{i}
=2^{n-2}+2(-1)^{\frac{r-1}{2}}(2^{\frac{n}{2}-1})\sin(2k \pi)
=2^{n-2}.
$$

\noindent Therefore, \begin{align*}
\sum_{i \equiv 1,2 \pmod{4}} \binom{n}{i}
-\sum_{i \equiv 0, 3 \pmod{4}} \binom{n}{i}
 &=\left(2^{n-2}+ ( 2^{n-2}-2^{\frac{n}{2}-1}) \right)- \left(   (2^{n-2} + 2^{\frac{n}{2}-1}) + 2^{n-2}  \right)\\ 
 &= 2(-2^{\frac{n}{2}-1})=-2^{\frac{n}{2}}.
\end{align*}
A similar reasoning applies to all remaining cases. 
\end{proof}

\subsection{Proof of Lemma~\ref{lem:jeig}.}

Recall in Lemma~\ref{lem:jeig} we have assume that $n$ is even
\jeig*

\begin{proof}
\allowdisplaybreaks

Starting with Equation~(\ref{equ:eigS}), we have that:
\begin{align}
\nonumber
\sum_{i \equiv 0,3 \Mod{4}} p_i(j) 
&=  \sum_{i \equiv 0,3 \Mod{4}} \left( \sum_{r=0}^j  (-1)^r  \binom{n-j}{i-r} \binom{j}{r}\right)=   \sum_{r=0}^j \left( \sum_{i \equiv 0,3 \Mod{4}}  (-1)^r \binom{n-j}{i-r} \binom{j}{r}\right) \\
\nonumber
&=   \sum_{i \equiv 0,3 \Mod{4}} \binom{n-j}{i} \binom{j}{0}  
- \sum_{i ~\equiv~ 0,3  \Mod{4}} \binom{n-j}{i-1} \binom{j}{1} + \sum_{i ~ \equiv ~ 0,3 \Mod{4}} \binom{n-j}{i-2} \binom{j}{2} \\
\nonumber
&
\quad + \dots + \sum_{i ~ \equiv ~ 0,3  \Mod{4}}  (-1)^j \binom{n-j}{i-j} \binom{j}{j}\\
\nonumber
&=\binom{j}{0} \sum_{k \equiv ~ 0,3 \Mod{4}} \binom{n-j}{k}  
-\binom{j}{1}  \sum_{k \equiv ~ 2,3 \Mod{4}} \binom{n-j}{k} +   \binom{j}{2}  \sum_{k ~ \equiv ~ 1,2 \Mod{4}} \binom{n-j}{k}   \\
&\quad -\binom{j}{3}  \sum_{k ~ \equiv ~ 0,1 \Mod{4}} \binom{n-j}{k}  + \cdots +  (-1)^j \binom{j}{j} \sum_{k~\equiv~ -j,~3-j \Mod{4}} \binom{n-j}{k}.\label{eq:last}
\end{align}

Next, we simplify the above expression and show that $|\lambda(S_U)|=2^{\frac{n-1}{2}}$. To do so, we consider different congruences modulo 8 of $n-j$. Computations for $n-j \equiv 0 ~ \pmod{8}$ are given below. All other cases follow a similar reasoning.  

If $n-j \equiv 0 ~ \pmod{8}$, then $j$ is also even since $n$ is even. We then have $ \mu_j(S_U)$ equal Equation~(\ref{eq:last}), which by Lemma~\ref{lem:sum0} equals

\begin{align*}
&\binom{j}{0} ( 2^{n-j-1} + 2^{(n-j)/2-1})   
-\binom{j}{1}  ( 2^{n-j-1} - 2^{(n-j)/2-1})   
 +   \binom{j}{2}  ( 2^{n-j-1} - 2^{(n-j)/2-1})  \\ 
&\quad    -  \binom{j}{3}  ( 2^{n-j-1} + 2^{(n-j)/2-1})  +
\dots
+  (-1)^j \binom{j}{j} ( 2^{n-j-1} \pm 2^{(n-j)/2-1})   \\
&= 2^{n-j-1}\left(\sum_{k'~\equiv ~ 0,2~\Mod{4}} \binom{j}{k'}-\sum_{k'~\equiv~1,3~\Mod{4}} \binom{j}{k'} \right)\\ \qquad 
& \quad - 2^{(n-j)/2 - 1} \left( \sum_{k'~\equiv ~ 0,1~\Mod{4}} \binom{j}{k'}-\sum_{k'~\equiv~2,3~\Mod{4}} \binom{j}{k'} \right)\\
&=2^{n-j-1} (0) -2^{(n-j)/2 - 1} \left( \sum_{k'~\equiv ~ 0,1~\Mod{4}} \binom{j}{k'}-\sum_{k'~\equiv~2,3~\Mod{4}} \binom{j}{k'} \right)\\
&= -2^{\frac{(n-j)}{2} - 1}(\pm 2^{\frac{j}{2}}) = 
\begin{cases}
2^{(n-1)/2}, & \textrm{ if } j \equiv 0, 6 \pmod{8}, \\
-2^{(n-1)/2}, & \textrm{ if }  j \equiv 2,4  \pmod{8}. 
\end{cases}
\end{align*}

Similar computations imply that $|\mu_j(S_U)|=2^{\frac{n-1}{2}}$ for the other congruency classes modulo 8 of $n-j$ and thus, it follows that
$$
\mu_j(S_U)=\sum_{i ~ \equiv~ 0,3 \Mod{4}} p_i(j) = \pm 2^{n/2-1}.
$$
As $\sum_{i} p_i(j) =0 $ and $U$ and $V$ partitions the set of relations of $H(n,2)$, then $\mu_j(T_V)=-\mu_j(S_U)$. The second equality then follows. 
\end{proof}

\subsection{Proof of Lemma~\ref{lem:valXY}.}

\valXY*

\begin{proof}
The valency $\mu_0(X)$ is the sum of the valencies of the graphs $X_i$ in $H(n,q)$ where $i$ is even.  If $n$ is even then this equals
\begin{align*}
\binom{n}{0}(q-1)^0  + \binom{n}{2}(q-1)^2 + \binom{n}{4 }(q-1)^4 + \dots + \binom{n}{n}(q-1)^n 
   = \frac{q^n + (q-2)^n}{2}.
\end{align*}
If $n$ is odd then it equals
\begin{align*}
\binom{n}{0}(q-1)^0  + \binom{n}{2}(q-1)^2 + \binom{n}{4 }(q-1)^4 + \dots + \binom{n}{n-1}(q-1)^{n-1} 
= \frac{q^n - (q-2)^n}{2}.
\end{align*}
Similarly, $\mu_0(Y)$ is be the sum of the valencies of the graphs $X_i$ for $i$ odd.
If $n$ is even then this is
\begin{align*}
\binom{n}{1}(q-1)^1  + \binom{n}{3}(q-1)^3 + \binom{n}{5}(q-1)^5 + \dots + \binom{n}{n-1}(q-1)^{n-1}
 =  \frac{q^n - (q-2)^n}{2},
\end{align*}
and if $n$ is odd this is
\begin{align*}
\binom{n}{1}(q-1)^1  + \binom{n}{3}(q-1)^3 + \binom{n}{5}(q-1)^5 + \dots + \binom{n}{n}(q-1)^{n} 
=  \frac{q^n + (q-2)^n}{2}.
\end{align*}
To verify the right-hand side of the above equations, use the binomial theorem twice:
\begin{align*}
\sum_{i=0}^n \binom{n}{i} (q-1)^i 1^n &= q^n, \\
 \sum_{i=0}^n \binom{n}{i} (q-1)^i (-1)^{n-i} &= ((q-1) - 1)^n = (q-2)^n.
\end{align*}

Summing these equations gives twice the sum of the binomial coefficients in which the lower number is even number; taking the difference gives twice the sum with the lower numbers is odd.
\end{proof}

\subsection{Proof of Lemma~\ref{lem:otherXYeigs}}

\otherXYeigs*

\begin{proof}
Consider
$$
	\sum_{i \textrm{ even}}p_i(j)= \sum_{i \textrm{ even}} \sum_{r=0}^{j} (-1)^r (q-1)^{i-r}\binom{n-j}{i-r}\binom{j}{r}.
$$
If $i<r$ then $\binom{n-j}{i-r} = 0$ so we can replace $i-r$ with $k$, where $k$ goes from 0 (with $i=r$) to $j$ (with $i=j$ and $r=0$). Since the term $\binom{n-j}{r-i} = \binom{n-j}{k}$ is equal to 0 whenever $k > n-j$, so we can assume that $k \leqslant  n-j$. Further, since $ \binom{j}{i-k} = 0$ if $4j < i-k$, we can consider $i-k$ going from $0$ to $j$.  It is known that $\sum_{i \textrm{ even}} \binom{j}{i-k} = 2^{j-1}$. With this we can rearrange the terms in this sum
\begin{align*}
	\sum_{i \textrm{ even}}p_i(j)
	&= \sum_{i \textrm{ even}} \sum_{r=0}^{j} (-1)^r (q-1)^{i-r}\binom{n-j}{i-r} \binom{j}{r} \\
	& = \sum_{i \textrm{ even}} \sum_{k=0}^{j} (-1)^{i-k} (q-1)^{k}\binom{n-j}{k} \binom{j}{i-k} \\
	& =  \sum_{k=0}^{n-j} (-1)^{-k} (q-1)^{k}  \binom{n-j}{k} \sum_{i \textrm{ even}}   \binom{j}{i-k} \\
	& =  \sum_{k=0}^{n-j} (-1)^{-k} (q-1)^{k}  \binom{n-j}{k} 2^{j-1} \\
	& =  2^{j-1}  (-1)^{n-j} \sum_{k=0}^{n-j}   \binom{n-j}{k} (-1)^{n-j-k} (q-1)^{k}.
\end{align*}
By the binomial theorem, we have
\[
 \sum_{k=0}^{n-j}   \binom{n-j}{k} (-1)^{n-j -k} (q-1)^{k}  = ( (q-1) + (-1) )^{n-j} = (q-2)^{n-j}.
 \]
Putting this together, we obtain
\begin{align*}
	\sum_{i \textrm{ even}}p_i(j) = 2^{j-1} (-1)^{n-j} (q-2)^{n-j}.\qedhere
\end{align*}
\end{proof}

\subsection{Proof of Lemma~\ref{lem:foldedEvalues}.}

\foldedEvalues*

\begin{proof}
\allowdisplaybreaks

Starting with Corollary~\ref{cor:Kraw2}, we have that:
\begin{align*}
\mu_i(S_U)
&=   \sum_{r=0}^j   \left( \sum_{i \equiv 0 \Mod{4}}  (-1)^r \binom{n-j}{i-r} \binom{j}{r}\right) \\
&=   \sum_{i \equiv 0 \Mod{4}} \binom{n-j}{i} \binom{j}{0}  
- \sum_{i \equiv  0 \Mod{4}} \binom{n-j}{i-1} \binom{j}{1} 
 + \sum_{i  \equiv 0 \Mod{4}} \binom{n-j}{i-2} \binom{j}{2} \\
& \quad + \dots 
+ \sum_{i ~ \equiv ~ 0  \Mod{4}}  (-1)^j \binom{n-j}{i-j} \binom{j}{j}\\
&=   \sum_{k \equiv 0 \Mod{4}} \binom{n-j}{k} \binom{j}{0}  
- \sum_{k \equiv  3 \Mod{4}} \binom{n-j}{k} \binom{j}{1} 
 + \sum_{k  \equiv 2 \Mod{4}} \binom{n-j}{k} \binom{j}{2} \\
& \quad + \dots 
+ \sum_{k \equiv -j \Mod{4}}  (-1)^j \binom{n-j}{k} \binom{j}{j} \\
&=
\begin{cases}
  2^{(n-j-1)/2-1} (  \binom{j}{0}     + \binom{j}{1}    -   \binom{j}{2}  - \binom{j}{3} +\dots  ) 
  & \textrm{ if } n \equiv 1 \pmod{8}, \\
   2^{(n-j-1)/2-1}  ( -  \binom{j}{0}   + \binom{j}{1}    +  \binom{j}{2}   - \binom{j}{3} - \dots )
   & \textrm{ if } n \equiv 3 \pmod{8}, \\
   2^{(n-j-1)/2-1}  ( -  \binom{j}{0}   - \binom{j}{1}    +  \binom{j}{2}   + \binom{j}{3} - \dots )
   & \textrm{ if } n \equiv 5 \pmod{8}, \\
   2^{(n-j-1)/2-1}  ( + \binom{j}{0}   - \binom{j}{1}    -  \binom{j}{2}   + \binom{j}{3} - \dots )
   & \textrm{ if } n \equiv 7 \pmod{8}, \\
 \end{cases}  
 \\
&=
\begin{cases}
2^{(n-3)/2}, & \textrm{ if } n-j \equiv 0, 2 \pmod{8}, \\
-2^{(n-3)/2}, & \textrm{ if }  n-j \equiv 4, 6  \pmod{8}.\\
\end{cases}
\end{align*}
The second to last equality follows from Lemma~\ref{lem:sum0}.
\end{proof}


\section{Table of results for Schurian association schemes}

\begin{longtable}{lcccl}
\mbox{Primitive Group:} & $d$ & $\lambda$ & $(s,t)$ & \mbox{Result or Section} \\ 
$(7,2)$ & 3 & 1& $(2,3)$ & Lemma~\ref{lem:cyclicGroup} \\ 
$(11,2)$ & 5 & 1& $(5,2)$ & Lemma~\ref{lem:cyclicGroup} \\ 
$(13,2)$ & 6 & 1& $(3,4)$ & Lemma~\ref{lem:cyclicGroup} \\ 
$(13,2)$  & 6  & 1 & $(2,6)$ & Lemma~\ref{lem:cyclicGroup} \\ 
$(13,3)$ & 4 & 3& $(4,9)$ & Proposition~\ref{prop:4ksquare9} \\
$(17,2)$ & 8 & 1& $(2,8)$ & Lemma~\ref{lem:cyclicGroup} \\ 
$(17,2)$ & 8 & 1& $(4,4)$ & Lemma~\ref{lem:cyclicGroup} \\ 
$(17,3)$ & 4 & 1& $(4,4)$ & Proposition~\ref{prop:16ksquared1}  \\ 
$(19,2)$ & 9 & 1 & $(3,6)$ & Lemma~\ref{lem:cyclicGroup} \\ 
$(19,2)$ & 9  & 1 & $(2,9)$ & Lemma~\ref{lem:cyclicGroup} \\ 
$(21,1)$ & 3 & 2& $(5,8)$ & see \S~\ref{subsec:ProjectiveGroups} \\
$(25,5)$ & 5 & 2& $(4,12)$ & see \S~\ref{subsec:schurianProduct}  and Theorem~\ref{thm:halfsetProds}  \\
$(27, [3,4,8] )$ & 3 & 4& $(8,13)$ & see \S~\ref{subsec:SchurianHamming} and \S~\ref{sec:Hamming} \\ 
$(28,1)$ & 4 & 2& $(18,3)$ & see \S~\ref{subsec:notFusions}  \\
$(28,1)$ & 4 & 4& $(18,6)$ & see \S~\ref{subsec:notFusions}  \\ 
$(31,4)^a$ & 6 & 4& $(6,20)$ & see \S~\ref{subsec:more4}  \\ 
$(31,4)^a$ & 6 & 8& $(16,15)$ & see \S~\ref{subsec:more4}  \\ 
$(37,4)$ & 9 & 1& $(4,9)$ & Lemma~\ref{lem:cyclicGroup}, see \S~\ref{subsec:isothinprime}  \\ 
$(37,6)^a$ & 4 & 7& $(9,28)$ & Proposition~\ref{prop:4ksquared1} \\ 
$(49,3)$ & 9 & 2& $(4,24)$ & see \S~\ref{subsec:schurianProduct} and Theorem~\ref{thm:halfsetProds} \\ 
$(49,12)^a$ & 5 & 8& $(16,24)$ & see \S~\ref{subsec:schurianProduct} and Theorem~\ref{thm:halfsetProds} \\ 
$(55,2)$ & 5 & 4& $(12,18)$ & see \S~\ref{subsec:ProjectiveGroups}\\ 
$(73,7)^a$ & 8 & 8& $(9,64)$ & Lemma~\ref{lem:64bsquared9}\\ 
$(81, [72,73,91] )$  & 4 & 8& $(16,40)$ & see \S~\ref{subsec:SchurianHamming} and  \S~\ref{sec:Hamming} \\ 
$(101,7)^a$  & 4 & 19 & $(25,76)$ & Corollary~\ref{prop:4ksquared1} \\ 
$(109,8)$ & 6 & 3 & $(18,18)$ & see \S~\ref{subsec:more4} and \cite{BaoEtAl2018} \\ 
$(109,8)$ & 6 & 27 & $(54,54)$ & see \S~\ref{subsec:more4}  \\ 
$(109,9)^a$ & 4 & 21 & $(28,81)$ & Proposition~\ref{prop:4ksquare9} \\ 
$(113,6)$ & 8 & 7 & $(28,28)$ & see \S~\ref{subsec:more4}  and \cite{li-momihara2025} \\ 
$(113,6)$ & 8 & 28 & $(56,56)$ & see \S~\ref{subsec:more4}  \\ 
$(121,29)^a$ & 5 & 18 & $(36,60)$ & see \S~\ref{subsec:schurianProduct} and Theorem~\ref{thm:halfsetProds}\\ 
$(125,[3,9])$ & 9 & 4 & $(8,62)$ & see \S~\ref{subsec:schurianProduct} and Theorem~\ref{thm:halfsetProds} \\ 
$(127,9)^a$  & 6 & 32 & $(65,42)$ & see \S~\ref{subsec:more4} \\ 
$(169,75)$ & 5 & 18 & $(36,84)$ & see \S~\ref{subsec:schurianProduct} and Theorem~\ref{thm:halfsetProds} \\ 
$(197,7)^a$ & 4 & 37 & $(49,148)$ & Proposition~\ref{prop:4ksquared1} \\ 
$(203,1)$ & 7 & 6 & $(12,101)$ & see \S~\ref{subsec:ProjectiveGroups}  \\ 
$(223,5)$ & 6 & 56 & $(111,112)$ & see \S~\ref{subsec:more4}  \\ 
$(225,[2,3,4,5])$ & 5 & 32 & $(64,112)$ & see \S~\ref{subsec:schurianProduct} and Theorem~\ref{thm:halfsetProds} \\ 
$(243, [13,15,21,22,23] )$ & 5 & 16 & $(32,121)$ & see \S~\ref{subsec:SchurianHamming} and \S~\ref{sec:Hamming} \\ 
$(257,13)$ & 4 & 16 & (64,64) & Proposition~\ref{prop:16ksquared1}  \\ 
$(289, [70,81] )$ & 4 & 8 & (48,48) & see \S~\ref{subsec:schurianProduct} \\ 
$(289, [70,81] )$ & 4 & 32 & (96,96) & see \S~\ref{subsec:schurianProduct}  \\ 
$(289,76)$ & 5 & 32 & (64,144) & see \S~\ref{subsec:schurianProduct} and Theorem~\ref{thm:halfsetProds}\\  

\caption{Minimal realizations of near-factorizations in small Schurian schemes of class $3 \leqslant d \leqslant 9$} \label{Table:Schurian}
\end{longtable}

\bibliographystyle{plain}
\bibliography{mybibfile}
\end{document}